\documentclass{amsart}

\usepackage[utf8]{inputenc}
\usepackage[T1]{fontenc}
\usepackage{amssymb}
\usepackage{amsmath}
\usepackage{amsfonts}
\usepackage{latexsym}
\usepackage{amsthm}
\usepackage{enumerate}
\usepackage{mathrsfs}
\usepackage{color}
\usepackage{footmisc}

\newtheorem{theo}{Theorem}[section]
\newtheorem{lemm}[theo]{Lemma}

\newtheorem{rema}[theo]{Remark}
\newtheorem{defi}[theo]{Definition}
\newtheorem{propo}[theo]{Proposition}
\newtheorem{coro}[theo]{Corollary}

\newtheorem{newconj}{Conjecture}[theo]

\newtheorem{newpropo}{Proposition}[theo]
\theoremstyle{definition}
\newtheorem{op}[theo]{Open problems}

\newcommand{\BigO}[1]{\ensuremath{\operatorname{O}\bigl(#1\bigr)}}
\newcommand{\betadRamb}{\widetilde{A}}
\newcommand{\betad}{A'}
\newcommand{\betar}{A}

\def\Z{\mathbb{Z}}
\def\N{\mathbb{N}}
\def\F{\mathbb{F}}

\def\Q{\mathbb{Q}}
\def\R{\mathbb{R}}
\def\C{\mathbb{C}}
\def\P{\mathbb{P}}
\def\p{\mathfrak{p}}
\def\cO{\mathcal{O}}

\newcommand{\Ld}[2][D]{\ensuremath{\mathcal{L}(#2\mathcal{#1})}} 
\newcommand{\D}[1][D]{\ensuremath{\mathcal{#1}}} 

\newcommand{\mus}{\mu^\mathrm{sym}}
\newcommand{\Ms}{M^\mathrm{sym}}
\newcommand{\ms}{m^\mathrm{sym}}

\begin{document}

\title[Tensor rank
of multiplication and related problems]{On the tensor rank
of multiplication in finite extensions of finite fields and related issues in algebraic geometry}

\author[Ballet]{St\'ephane Ballet}
\address{Aix-Marseille Université, CNRS, Centrale Marseille, Institut de Mathématiques de Marseille\\ 
case 907, 163 avenue de Luminy, F13288 Marseille cedex 9\\ France}
\email{stephane.ballet@univ-amu.fr}

\author[Chaumine]{Jean Chaumine}
\address{Laboratoire G\'eom\'etrie Alg\'ebrique et 
Applications \`a la Th\'eorie de l'Information\\ 
Universit\'e de la Polyn\'esie Fran\c{c}aise\\ 
B.P. 6570, 98702 Faa'a, Tahiti\\ France}
\email{jean.chaumine@upf.pf}

\author[Pieltant]{Julia Pieltant}
\address{Conservatoire National des Arts et Métiers \\Équipe en émergence Sécurité-Défense\\ 
EPN 15 STRATÉGIES\\
Pôle Sécurité Défense - Chaire de Criminologie\\
40 rue des Jeûneurs, F75002 Paris\\ France}
\email{julia.pieltant@lecnam.net}

\author[Rambaud]{Matthieu Rambaud}
\address{CNRS LCTI\\
Télécom ParisTech,
46 rue Barrault,  F-75634 Paris cedex 13\\ France}
\email{rambaud@enst.fr}

\author[Randriambololona]{Hugues Randriambololona}
\address{CNRS LCTI\\
Télécom ParisTech,
46 rue Barrault,  F-75634 Paris cedex 13\\ France}
\email{randriambololona@enst.fr}

\author[Rolland]{Robert Rolland}
\address{Aix-Marseille Université, CNRS, Centrale Marseille, Institut de Mathématiques de Marseille\\ 
case 907, 163 avenue de Luminy, F13288 Marseille cedex 9\\ France}
\email{robert.rolland@acrypta.fr}

\date{\today}

\begin{abstract}
In this paper, we give a survey of the known results concerning the tensor rank of the multiplication 
in finite extensions of finite fields, enriched with some not published recent results as well as analyzes enhancing 
the qualitative understanding of the domain. 
In particular, we identify and clarify certain results not completely proved and we emphasis the link with open problems 
in number theory, algebraic geometry, and coding theory. 
\end{abstract}

\maketitle

\tableofcontents 

%
%
%
%
%
%
%
%
%
%
%
%
%

\section{Introduction} \label{SecIntro}

This article proposes a survey on the tensor rank of the multiplication in finite fields. It is an update 
of the previous survey \cite{baro2} published about ten years ago. The deep improvements done since then require 
a complete rewrite of the survey highlighting the current state of the art. In particular, 
we present the new techniques introduced 
in recent years. The growing importance of this topic has attracted many mathematicians and computer scientists 
who developed new ideas and obtained new results. At the same time, we report a number of non-trivial errors and solutions
which testify to the vividness of the domain and the community concerned.
The finite fields are an important area. They arise in many fields
applications, particularly in areas related to information theory.
In particular, the complexity of the multiplication in the finite fields is a central problem.
It is a part of the algebraic complexity for which the best general reference is \cite{buclsh}. 
It turns out that studying this problem has raised many issues of number theory and algebraic geometry. 
Notably, it has revealed deep links between these different domains. So, one of the objectives of this article 
is also to explicit these links and to present current related open problems. In the same time
we prove some new results not yet published.

Let us describe more precisely the problem: we suppose that we have the multiplication 
in a finite field ${\mathbb F}_q$ and we want to construct an algorithm of multiplication in 
the extension ${\mathbb F}_{q^n}$ which is the least expansive in terms of operations in ${\mathbb F}_q$.
Let us remark that from this point of view the multiplication in ${\mathbb F}_{q^n}$ is the multiplication 
of two polynomials of degree $<n$ with coefficients in ${\mathbb F}_q$. We then distinguish in the algorithm 
two types of operations: those which are linear with respect to the variables that one multiply and those which are 
bilinear with respect to the two variables. 
More precisely, let $\mathcal{B}=\{e_1,...,e_n\}$ be a basis of ${\mathbb F}_{q^n}$ over ${\mathbb F}_q$.
If $x=\sum_{i=1}^{n}x_ie_i$ et $y=\sum_{i=1}^{n}y_ie_i$ then a direct computation gives:
\begin{equation}\label{calculdirect}
z=xy=\sum_{h=1}^{n}z_he_h=\sum_{h=1}^{n}\biggr( \sum_{i,j=1}^{n}t_{ijh}x_ix_j\biggl)e_h,
\end{equation}
where $$e_ie_j=\sum_{h=1}^{n}t_{ijh}e_h,$$ 
$t_{ijh}\in {\mathbb F}_q$ being constants.
Then the problem of the algebraic complexity consists on determining the minimal number of 
elementary operations in ${\mathbb F}_q$ 
required to compute the product of two elements  $x,y \in {\mathbb F}_{q^n}$. We can distinguish the following operations:
\begin{itemize}
\item addition : $(\alpha,\beta) \mapsto \alpha+\beta$ où $\alpha,\beta \in {\mathbb F}_q$,
\item scalar multiplication : $x_i \mapsto \alpha \cdot x_i$ where $\alpha, x_i \in {\mathbb F}_q$, 
and $\alpha$ is a constant,
\item non-scalar or bilinear multiplication :  $(x_i,y_j) \mapsto x_i \cdot y_j$ where $x_i,y_j \in {\mathbb F}_q$ 
depend on the elements 
$x$ and $y$ of $\mathbb{F}_{q^n}$ which are multiplied.
\end{itemize}
So, to obtain the product $xy$ by the direct computation, one counts: 
\begin{itemize}
\item $n^3-n$ additions, 
\item$n^3$ scalar multiplications, 
\item $n^2$ non-scalar or bilinear multiplications.
\end{itemize}

The bilinear complexity of the algorithm of multiplication is given by the number of used bilinear multiplications.
This complexity corresponds to the rank of the tensor of multiplication corresponding to this algorithm in ${\mathbb F}_{q^n}$ 
as vector space over ${\mathbb F}_{q}$, as will be explained in the next section. 

The bilinear complexity of multiplication in finite fields  ${\mathbb F}_{q^n}$ 
over  ${\mathbb F}_{q}$ is obtained by a tensor (resp. an algorithm) of minimal rank (resp. of minimal bilinear complexity).
The survey emphases the study of this minimal complexity. 

In this paper, it is a question of introducing the problem of the tensor rank of the multiplication 
in finite fields and of giving  a statement of the results obtained in this
part of algebraic complexity theory, as well as related issues. 

\subsection{Tensor rank and multiplication algorithm}

Let us recall the notions of multiplication algorithm and associated bilinear complexity.  

\begin{defi}
Let $K$ be a field and $E_0, \ldots , E_s$ be finite dimensional $K$-vector spaces. 
A non zero element $t\in E_0 \otimes \cdots \otimes E_s$ 
is said to be an elementary tensor, or a tensor of rank 1, if it can be written in 
the form $t=e_0 \otimes \cdots \otimes e_s$ for some $e_i \in E_i$. 
More generally, the rank of an arbitrary $t\in E_0 \otimes \cdots \otimes E_s$  is 
defined as the minimal length of a decomposition of $t$ as a sum of elementary tensors.
\end{defi}

\begin{defi}
If $$ \alpha~~:~~E_1\times \cdots \times E_s \longrightarrow E_0$$
is an $s$-linear map, the $s$-linear complexity of $\alpha$ is defined as the tensor rank of the element
$$\tilde{\alpha}\in E_0\otimes E_1^{\vee} \otimes \cdots \otimes E_s^{\vee}$$ 
where $E_i^{\vee}$ denotes the dual of $E_i$ 
as vector space over $K$ for any integer $i$, naturally deduced from $\alpha$. In particular, 
the $2$-linear complexity is called the bilinear complexity.
\end{defi}

\begin{defi}
Let $\mathcal A $ be a finite-dimensional $K$-algebra. We denote by 
$$\mu(\mathcal A/K)$$
the bilinear complexity of the multiplication map
$$m_{\mathcal A}~~:~~\mathcal A \times \mathcal A \longrightarrow \mathcal A$$
considered as a $K$-bilinear map. 

\noindent
In particular, if $\mathcal{A}=\F_{q^m}$ and $K=\F_q$,  we set: $$\mu_q(m)=\mu(\F_{q^m}/\F_q).$$
\end{defi}

More concretely, $\mu(\mathcal A/K)$ is the smallest integer $n$ such that there exist linear forms 
$\phi_1,\ldots , \phi_n, \psi_1, \ldots  , \psi_n~~:~~\mathcal A \longrightarrow K$, 
and elements $w_1, \ldots , w_n \in \mathcal A$, 
such that for all $x,y\in \mathcal A$ one has
\begin{equation} xy= \phi_1(x)\psi_1(y)w_1+\cdots +\phi_n(x)\psi_n(y)w_n, \label{xy}\end{equation}
since such an expression is the same thing as a decomposition
\begin{equation}
t_M=\sum_{i=1}^{n}w_i\otimes \phi_i\otimes \psi_i \in  \mathcal{A} \otimes \mathcal{A}^{\vee} \otimes \mathcal{A}^{\vee}
\end{equation}
for the multiplication tensor of $\mathcal{A}$.

\begin{defi}
We call multiplication algorithm of length $n$  for $\mathcal A/K$ 
a collection of $\phi_i, \psi_i, w_i$ that satisfy \eqref{xy} or equivalently a tensor decomposition 
$$
t_M=\sum_{i=1}^{n}w_i\otimes \phi_i\otimes \psi_i \in  \mathcal{A} \otimes \mathcal{A}^{\vee} \otimes \mathcal{A}^{\vee}
$$ 
for the multiplication tensor of $\mathcal{A}$. Such an algorithm is said 
symmetric if $\phi_i=\psi_i$ for all $i$ (this can happen only if $\mathcal A$ is commutative).
\end{defi}

Hence, when $\mathcal{A}$ is commutative, it is interesting to study the minimal length 
of a symmetric multiplication algorithm. 

\begin{defi}
Let $\mathcal A $ be a finite-dimensional commutative $K$-algebra. The symmetric bilinear complexity $$\mus(\mathcal A/K)$$ 
is the minimal length of a symmetric multiplication algorithm. 

\noindent
In particular, if $\mathcal{A}=\F_{q^m}$ and $K=\F_q$,  we set: $$\mus_q(m)=\mus(\F_{q^m}/\F_q).$$
\end{defi}

Here are some basic properties of these quantities, taken from \cite[Lemma 1.10]{randJComp}:
\begin{lemm}
\label{functorial_inequalities}
\begin{enumerate}[(a)]
\item
If $\mathcal{A}$ is a finite-dimensional $K$-algebra
and $L$ an extension field of $K$,
and if we let $\mathcal{A}_L=\mathcal{A}\otimes_K L$ considered as an $L$-algebra, then
\begin{equation*}
\label{funct1}
\mu(\mathcal{A}_L/L)\leq\mu(\mathcal{A}/K).
\end{equation*}
Moreover, if $\mathcal{A}$ is commutative, we also have
\begin{equation*}
\label{funct1s}
\mus(\mathcal{A}_L/L)\leq\mus(\mathcal{A}/K).
\end{equation*}
\item
If $\mathcal{A}$ is a finite-dimensional $L$-algebra,
where $L$ is an extension field of $K$,
then $\mathcal{A}$ can also be considered as a $K$-algebra, and
\begin{equation*}
\label{funct2}
\mu(\mathcal{A}/K)\leq\mu(\mathcal{A}/L)\mu(L/K).
\end{equation*}
Moreover, if $\mathcal{A}$ is commutative, we also have
\begin{equation*}
\label{funct2s}
\mus(\mathcal{A}/K)\leq\mus(\mathcal{A}/L)\mus(L/K).
\end{equation*}
\item
If $\mathcal{A}$ and $\mathcal{B}$ are two finite-dimensional $K$-algebras,
\begin{equation*}
\label{funct3}
\mu(\mathcal{A}\times\mathcal{B}/K)\leq\mu(\mathcal{A}/K)+\mu(\mathcal{B}/K).
\end{equation*}
Moreover, if $\mathcal{A}$ and $\mathcal{B}$ are commutative, we also have
\begin{equation*}
\label{funct3s}
\mus(\mathcal{A}\times\mathcal{B}/K)\leq\mus(\mathcal{A}/K)+\mus(\mathcal{B}/K).
\end{equation*}
\item
If $\mathcal{A}$ and $\mathcal{B}$ are two finite-dimensional $K$-algebras,
\begin{equation*}
\label{funct4}
\mu(\mathcal{A}\otimes_K\mathcal{B}/K)\leq\mu(\mathcal{A}/K)\mu(\mathcal{B}/K).
\end{equation*}
Moreover, if $\mathcal{A}$ and $\mathcal{B}$ are commutative, we also have
\begin{equation*}
\label{funct4s}
\mus(\mathcal{A}\otimes_K\mathcal{B}/K)\leq\mus(\mathcal{A}/K)\mu(\mathcal{B}/K).
\end{equation*}
\end{enumerate}
\end{lemm}
In particular, the following lemma of Shparlinski, Tsfasman, and Vladut 
\cite[Lemma 1.2]{shtsvl}, is especially useful.
Actually, the right-hand inequality was already stated in the original paper of 
D.V. Chudnovsky and G.V.Chudnovsky \cite[eq. (6.2)]{chch},
so the new contribution of I. Shparlinski, M. Tsfasman, and S. Vladut is the left-hand inequality.
This will be important when we will consider asymptotic complexities in Lemma \ref{lemasyMqmq}.
\begin{lemm}\label{plongement}
For all $m,n$ we have
$$\mu_{q}(n)\leq\mu_q(mn)\leq \mu_q(m)\cdot\mu_{q^m}(n).$$
\end{lemm}
Actually the same holds for symmetric complexity.
\begin{lemm}\label{plongementsym}
For all $m,n$ we have
$$\mus_{q}(n)\leq\mus_q(mn)\leq \mus_q(m)\cdot\mus_{q^m}(n).$$
\end{lemm}
\begin{proof}
The left-hand inequalities $\mu_{q}(n)\leq\mu_q(mn)$ and $\mus_{q}(n)\leq\mus_q(mn)$ are 
consequences of the inclusion $\F_{q^n}\subseteq\F_{q^{mn}}$.
Then, for the right-hand inequalities $\mu_q(mn)\leq \mu_q(m)\cdot\mu_{q^m}(n)$ and $\mus_q(mn)\leq \mus_q(m)\cdot\mus_{q^m}(n)$,
we apply Lemma~\ref{functorial_inequalities}(b) with $\mathcal{A}=\F_{q^{mn}}$, $L=\F_{q^m}$, and $K=\F_q$.
\end{proof}

\subsection{Organization of the paper} 

In Section \ref{SecOld}, we present the classical results via the approach 
using the multiplication by polynomial interpolation.
In Section \ref{SectionAlgCurvApproach}, we give an historical record of 
results obtained from the pioneer works due to 
D.V. and G.V. Chudnovsky in \cite{chch} and later I. Shparlinski, M. Tsfasman and S. Vladut in \cite{shtsvl}. 
In particular, we present the original algorithm. 
This modern approach uses the interpolation over algebraic curves defined over 
finite fields. This approach, which we recount the first success 
as well as the rocks on which the pionners came to grief, enables to end at a first 
complete proof of the linearity of 
the bilinear complexity of multiplication by S. Ballet in \cite{ball1}. 
In Section \ref{SecCodes}, we present the code approach for the bilinear complexity 
and explain the connexion between the bilinear complexity of multiplication and 
the so-called (exact) supercodes, or equivalently multiplication friendly 
codes in the lexicon of certain authors. 
Then, in Section \ref{SecGeneralizations}, we present the different generalizations 
of the original D.V. and G.V. Chudnovsky algorithm, in particular the most successful 
version of the algorithm of Chudnovsky--Chudnovsky type at the present time, due to 
H. Randriambololona in \cite{randJComp}. This part explains the links with algebraic geometry.
In Section \ref{AsyBounds}, we recall  the known results on the asymptotic bounds about 
the symmetric and asymmetric bilinear complexity that have been established through the last 30 years. 
Then, in a same way, in Section \ref{UniBounds}, we give uniform bounds about the symmetric 
and asymmetric bilinear complexity.
Finally, in Section \ref{SecEffective} we present methods about the effective construction 
of bilinear multiplication algorithms in finite fields.

\section{Old classical results} \label{SecOld} 
Let 
$$P(u)=\sum_{i=0}^{n}a_iu^i$$
be a  monic irreducible polynomial of degree $n$ with
coefficients in a field $F$.
Let
$$R(u)=\sum_{i=0}^{n-1}x_iu^i$$
and
$$S(u)=\sum_{i=0}^{n-1}y_iu^i$$
be two polynomials of degree $\leq n-1$ where
the coefficients $x_i$ and $y_i$ are indeterminates.

C. Fiduccia and Y. Zalcstein (cf. \cite{fiza}, \cite{buclsh}  p.367 Prop. 14.47)
have studied the general problem of computing
the coefficients of the product $R(u) \times S(u)$
and they have shown that at least $2n-1$ multiplications
are needed.
When the field $F$ is infinite, an algorithm reaching exactly
this bound was previously given by A. Toom in \cite{toom}.
S. Winograd described in \cite{wino2} all the algorithms reaching
the bound $2n-1$.
Moreover, S. Winograd proved in \cite{wino3} that up to some
transformations every algorithm for computing the coefficients
of $R(u) \times S(u) \mod P(u)$ which is of bilinear complexity $2n-1$,
necessarily computes the coefficients of $R(u) \times S(u)$, and consequently
uses one of the algorithms described in \cite{wino2}.
These algorithms use interpolation techniques and cannot be performed
if the cardinality of the field $F$ is $<2n-2$. In conclusion,
we have the following result:

\begin{theo}\label{old}
If the cardinality of $F$ is $<2n-2$, every algorithm computing
the coefficients of  $R(u) \times S(u) \mod P(u)$ has a bilinear
complexity $>2n-1$.
\end{theo}

Applying the results of S. Winograd and H. De Groote \cite{groo} and Theorem \ref{old} to
the multiplication in a finite extension $\F_{q^n}$ of a finite field
$\F_q$ we obtain:

\begin{theo}\label{thm_wdg}
The bilinear complexity $\mu_q(n)$ of the multiplication
in the finite field $\F_{q^n}$ over $\F_q$ verifies
$$\mu_q(n) \geq 2n-1,$$
with equality holding if and only if
$$n \leq \frac{q}{2}+1.$$
\end{theo}

This result does not give any estimate of an upper bound for 
$\mu_q(n)$, when $n$ is large.
In \cite{lesewi}, A. Lempel, G. Seroussi and S. Winograd proved
that $\mu_q(n)$ has a quasi-linear upper bound.
More precisely:

\begin{theo}
The bilinear complexity of the multiplication
in the finite field $\F_{q^n}$ over $\F_q$ verifies:
$$\mu_q(n) \leq f_q(n)n,$$
where $f_q(n)$ is a very slowly growing function defined recursively by 
$$f_q(n)=2f_q((\lceil\log_q(2(q-1)n) \rceil),$$ $n\geq 4$, $q\geq 2$.\\ 
For $n<4$, $f_q(n)$ is defined as follows:

$$f_q(n)=\left\{
\begin{array}{ll}
 1,&  n=1, \hbox{ } q\geq 2, \\
 &   \\
  \dfrac{3}{2},&  n=2, \hbox{ } q\geq 2,\\
  &   \\
  \dfrac{5}{3},&  n=3, \hbox{ } q\geq 4,\\  
  &   \\
  2,&  n=3, \hbox{ } 2\leq q\leq 3. 
\end{array}\right .$$
\end{theo}

\begin{coro}
Asymptotically, 
$$f_q(n)<\underbrace{\log_q\log_q \cdots \log_q}_{\mbox{$k$ times}}(n)$$
for any $k \geq 1$.
\end{coro}

Furthermore, extending and using more efficiently the technique
developed in \cite{bska1}, N. Bshouty and M. Kaminski
showed that
$$\mu_q(n) \geq 3n-o(n)$$
for $q \geq 3.$
The proof of the above lower bound on the complexity of straight-line
algorithms for polynomial multiplication is based on the analysis
of Hankel matrices representing bilinear forms defined by linear
combinations of the coefficients of the polynomial product.

\section{The approach via algebraic curves}\label{SectionAlgCurvApproach}

We have seen in the previous section
that if the number of points of the ground field is too low, we cannot perform
the multiplication by the Winograd interpolation method. 
D.V. and G.V. Chudnovsky
have designed in \cite{chch} an algorithm 
where the interpolation is done on points of an 
algebraic
curve over the groundfield with a sufficient number of rational points. 
We will denote by CCMA this Chudnovsky--Chudnovsky Multiplication Algorithm.
Using this algorithm, D.V. and G.V. Chudnovsky claimed that 
the bilinear complexity of
the multiplication in finite extensions of a finite field is asymptotically linear but 
later I. Shparlinski, M. Tsfasman and S. Vladut in \cite{shtsvl}  
noted that they only proved that the quantity $m_q=\liminf_{k \rightarrow \infty}\frac{\mu_q(k)}{k}$ is bounded 
which does not enable to prove the linearity. To prove the linearity, it is also necessary to prove that 
$M_q= \limsup_{k \rightarrow \infty}\frac{\mu_q(k)}{k}$ is bounded which is the main aim of their paper. 
However,  I. Cascudo, R. Cramer and C. Xing recently detected a mistake in the proof 
of I. Shparlinski, M. Tsfasman and S. Vladut. 
Unfortunately, this mistake that we will explain in details in this section, also had an 
effect on their improved estimations of $m_q$. 

After the above pioneer research, S. Ballet obtained in \cite{ball1} 
(cf. also \cite{ball0}) the first upper bounds uniformly with respect to $q$ 
for $\mu_q(n)$. The algorithm CCMA being clearly symmetric, these first uniform bounds 
also concerned $\mus_q(n)$. 
Moreover, these bounds not being affected by the same mistake enable at the same time 
to prove the linearity of the bilinear complexity of
the multiplication in finite extensions of a finite field since it obviously implied 
that $M_q$ was finite. 
Subsequently, critical improvements were introduced: in \cite{ball0}\cite{ball1}, S. Ballet 
introduces simple numerical 
conditions on algebraic curves of an arbitrary genus $g$ giving a sufficient condition 
for the application of the algorithm CCMA 
(existence of places of certain degree, of non-special divisors of degree $g-1$) 
generalizing the result of A. Shokrollahi \cite{shok} for the elliptic curves; in 
\cite{ball0}\cite{ball1} S. Ballet introduces the use of  towers 
of algebraic functions fields and their densification in \cite{ball3}; in \cite{baro1} S. Ballet and R. Rolland introduce
the use of places of higher degree; in \cite{baro1} S. Ballet and R. Rolland introduce 
the descent over $\F_q$ of the definition field $\F_{q^2}$ of a densified 
tower defined over $\F_{q^2}$ for any finite field $\F_q$ with a 
caracteristic $p=2$ and in \cite{balbro}, S. Ballet, D. Le Brigand and R. Rolland generalize 
the method for any finite field; in \cite{ball4} , S. Ballet derive optimal criterions for direct construction 
of the divisors satisfying the needed conditions and 
in \cite{chau0}\cite{chau1}, J. Chaumine proves that these criterions 
are always satisfied in the elliptic case, 
so improving the result of A. Shokrollahi \cite{shok}; in \cite{balb}, 
thanks to an existence theorem of non-special divisors of degree $g-1$, S. Ballet and D. Le Brigand 
improve sufficient conditions for the application of the algorithm CCMA for the extensions of arbitrary finite fields; 
in \cite{arna1}, N. Arnaud introduces the use of local expansion, called derivated evaluation; 
in \cite{bapi} \cite{piel} S. Ballet and Julia Pieltant introduce the use of divisors of 
degree zero thanks to a existence result obtained in \cite{bariro} 
by S. Ballet, C. Ritzenthaler and R. Rolland and combine it with local expansion.
Then M. Cenk and F. \"Ozbudak \cite{ceoz}, and H. Randriambololona \cite{randJComp} gave improvements by
using of local expansion and high degree places.
These can be combined with the following other independent ingredients, also proposed in \cite{randJComp}:
allowing asymmetry in the interpolation procedure,
which establishes the announced Shparlinski-Tsfasman-Vladut estimates for $m_q$ and $M_q$;
and using the best bilinear complexities recursively, an idea that was then also used in \cite{babotu}.
Last, two ideas can be used in order to deal with symmetric complexities:
bounds involving the $2$-torsion \cite{Xing}\cite{randITW}\cite{cacrxi}\cite{cacrxi2},
and direct construction of the divisors satisfying the needed conditions \cite{randIJM}\cite{rand2D-G}\cite{randJComp}.
Ultimately this allows to obtain for most cases the Shparlinski-Tsfasman-Vladut estimates also for $\ms_q$ and $\Ms_q$,
as well as other related estimates for symmetric complexity.

\subsection{The D.V. Chudnovsky and G.V. Chudnovsky algorithm (CCMA)}

In this section, we recall the brilliant idea of  D.V. Chudnovsky and G.V. Chudnovsky and give their main result.
First, we present  the original CCMA, 
which was established in 1987 in \cite{chch}.

\begin{theo}\label{chudchud}
Let 
\begin{itemize}
\item $F/\F_q$ be an algebraic function field,
\item $Q$ be a degree $n$ place of $F/\F_q$, 
\item ${\D}$ be a divisor of $F/\F_q$, 
\item ${\mathscr P}=\{P_1,\ldots,P_{N}\}$ be a set of places of degree $1$.
\end{itemize}
We suppose that $Q$, $P_1, \ldots, P_N$ are not in the support of ${\D}$
and that:
\begin{enumerate}[(a)]
	\item the evaluation map 
	$$
	Ev_Q: 
	\left |
	\begin{array}{cll}
  	\Ld{}  &\rightarrow & \F_{q^n}\simeq F_Q \\
   	f & \mapsto & f(Q) \\
	\end{array}
	\right .
	$$
	is onto (where $F_Q$ is the residue class field of $Q$),
	\item the application 
	$$
	Ev_{{\mathscr P}}:
	\left |
	\begin{array}{cll}
  	\Ld{2} & \rightarrow & \F_{q}^{N}  \\
   	f & \mapsto & (f(P_1),\ldots,f(P_{N})) \\
	\end{array}
	\right .
	$$
	is injective.
\end{enumerate}
Then  $$\mu_q(n)\leq N.$$
\end{theo}
We presented this result as it was formulated in \cite{chch}, in terms of the bilinear complexity $\mu_q(n)$.
However closer inspection of the method shows that it produces symmetric algorithms, 
so the conclusion also holds for the \emph{symmetric} bilinear complexity:
$$\mus_q(n)\leq N.$$

\subsection{The linearity of the bilinear complexity of the multiplication}\label{mM}

As seen previously, I. Shparlinski, M. Tsfasman and S. Vladut have given in \cite{shtsvl} many interesting remarks on CCMA and the bilinear complexity.
In particular, they have considered asymptotic bounds$\footnote{The families of curves used by the pioneers only gave asymptotic bounds. M. Tsfasman in a private communication asked for the question of finding uniform bounds to R. Rolland. \label{borunivTsfa}}$
for the bilinear complexity in order to prove the linearity of this complexity 
from CCMA.
Following these authors, let us define 
$$M_q= \limsup_{k \rightarrow \infty}\frac{\mu_q(k)}{k}$$
and
$$m_q=\liminf_{k \rightarrow \infty}\frac{\mu_q(k)}{k}.$$

Moreover, we also have to consider the symmetric variants of these quantities which 
were not considered 
by I. Shparlinski, M. Tsfasman and S. Vladut,
but were first introduced by 
H. Randriambololona in \cite{randJComp}, and have become equally important since then: 

$$\Ms_q= \limsup_{k \rightarrow \infty}\frac{\mus_q(k)}{k}$$
and
$$\ms_q=\liminf_{k \rightarrow \infty}\frac{\mus_q(k)}{k}.$$

It is clear that we have: $$M_q\leq \Ms_q$$ and $$m_q\leq \ms_q.$$
 
It is not obvious at all that either of these values is finite. Note that if $M_q$ (resp. $\Ms_q$) is finite, then bilinear complexity (resp. the symmetric bilinear complexity) of multiplication is linear in the degree of extension, namely there exists a constant $C_q\geq M_q$ (resp. $C^{sym}_q\geq \Ms_q$) such that for any integer $n>1$, 
$$\mu_q(n)\leq C_qn \quad (\hbox{resp. } \mus_q(n)\leq C^{sym}_qn).$$ 


From Theorem \ref{chudchud}, D.V. Chudnovsky and G.V. Chudnovsky derive \cite[Theorem 7.7]{chch}\footnote{This result is originally formulated for $\mu_q(n)$. Although at this time most authors did not distinguish in the notation between bilinear complexity and symmetric bilinear complexity, it was known that the CCMA naturally produces symmetric algorithms (cf. \cite[Definition p. 154 and Remark 2.2]{chch} and also more precisely \cite[Proof of Theorem 1.1]{ball1}), so the estimate also holds for the symmetric bilinear complexity $\mus_q(n)$.\label{symornot}}: 
for $q\geq25$ a square, as $n\to\infty$, we have 
\begin{equation}\label{chudmq}
\mus_q(n)\leq 2\left(1+\frac{1}{\sqrt{q}-3}\right)\cdot n+o(n).
\end{equation}

However, as pointed out by I. Shparlinski, M. Tsfasman and S. Vladut, the proof given for Bound (\ref{chudmq}) is quite sketchy, with some important details missing.
This made them question its validity.

More precisely, relying on Ihara's work \cite{ihar}, D.V. Chudnovsky and G.V. Chudnovsky consider 
Shimura modular curves having an asymptotically maximal number of points over $\F_q$,
and in the final step of their argument, they assert that, for some given constant $C$
and for all integers $n$ large enough, they can choose curves in this family of genus $g=C\cdot n+o(n)$.
Although it follows from \cite{ihar} that this is possible for infinitely many $n$,
D.V. Chudnovsky and G.V. Chudnovsky need it to hold for \emph{all} $n$, for which they do not give justification.
Because of this, I. Shparlinski, M. Tsfasman and S. Vladut explain that one should consider that,
although D.V. Chudnovsky and G.V. Chudnovsky state an estimate for the limsup $M_q$, their proof is valid only 
for the liminf $m_q$.

But then, with \cite[Claim, p.~163]{shtsvl}, I. Shparlinski, M. Tsfasman and S. Vladut
precisely describe a family of Shimura curves that satisfy the conditions needed by D.V. Chudnovsky and G.V. Chudnovsky,
which essentially completes the proof of \eqref{chudmq}.
Unfortunately, at the same time,  I. Shparlinski, M. Tsfasman and S. Vladut also propose
to replace \eqref{chudmq} with a sharper bound,
and in doing so they
introduce in the proof an unproved argument. The gap in their proof was found by I. Cascudo, R. Cramer and C. Xing 
(cf. personal communication in 2009 and \cite[Section V]{cacrxi2}). They present the gap as follows:
\textit{the mistake in 
\cite{shtsvl}  from 1992 is in the proof of their Lemma 3.3, 
page 161, the paragraph following formulas about the degrees of the divisor. It reads: ``\,}\textsl{Thus the  number 
of linear equivalence classes of degree $a$ for which either 
Condition $\alpha$ or Condition $\beta$ fails is at most $D_{b'} + D_b$.}\textit {'' This is incorrect; $D_b$ should 
be multiplied by the torsion. Hence the proof of their asympotic bound is incorrect.}~\guillemotright. 
Note that a synthesis work enabling to fill the gap 
let in the proof of D.V and G. V. Chudnovsky with the approach of Shparlinski, Tsfasman and Vladut is possible but not direct. 
Anyway, independently, 
by using the strategy of D.V and G. V. Chudnovsky applied to the first 
tower\footnote{The advantage of this tower of algebraic function fields is that firstly 
one knows explicitly the number of rational points and the the genus for each step, 
secondly the ratio of rational points over the genus is very good.} of 
Garcia-Stichtenoth \cite{gast} attaining the Drinfeld-Vladut bound, 
joint to a result concerning the existence of non-special divisors of degree $g-1$, 
S. Ballet gives in \cite{ball1} the first complete proof of the linearity of 
the bilinear complexity of the multiplication. More precisely, it was done by 
determining directly upper bounds for $C^{sym}_q$. From there, different works 
were done to improve the asymptotic bounds (cf. Section \ref{AsyBounds}) and the 
uniform bounds (cf. Section \ref{UniBounds}).

\section{The approach via codes}\label{SecCodes}

Initially, just after the pioneer work of D.V. and G.V. Chudnovsky \cite{chch}, 
I. Shparlinski, M. Tsfasman and S. Vladut in \cite{shtsvl} 
specified the link between certain codes and multiplication tensors. 
Then, they introduced the notion of exact supercodes also called multiplication friendly codes.

\subsection{Connection with codes and asymptotic lower bounds}\label{connectcode}

First, let us recall the link between the linear error-correcting codes and  
the decomposition of multiplication tensors.

Let us recall the following classical definition:

\begin{defi}
A linear error-correcting code $C$ over $\F_q$ of length $N$, dimension $n$ and 
Hamming distance $d$ is called an $[N,n,d]_q$-code. 
The rate $\frac{n}{d}$  of such a code is denoted by $R$ and its relative minimum 
distance $\frac{d}{n}$ by 
$\delta$.
\end{defi}

By \cite{shtsvl}, it is possible to construct a code using decomposition of $t_M$ into a 
sum of rank one tensors. Indeed, if 
$$t_M = \sum^{N}_{l=1}a_l \otimes b_l \otimes c_l$$ where $a_l \in  \F_{q^n}^*$, $b_l \in  \F_{q^n}^*$, $c_l \in  \F_{q^n}$, 
then one defines an $\F_q$-linear map 
$$
\begin{array}{cccc}
\phi: & \F_{q^n} & \longrightarrow & \F_q^N  \\
       & x  & \longmapsto & (a_1(x),\ldots,a_N(x)).
\end{array}
$$

From  \cite{shtsvl}, it follows that:

\begin{propo} \label{code}
The $\F_q$-vector space $C=\mathrm{Im}\, \phi$ is an $[N,n,d]_q$-code such that $d\geq n$. 
\end{propo}

\begin{coro} \label{tensorcode}
Any  decomposition of length $N$ of a tensor of multiplication in the finite field 
$\F_{q^n}$ gives an $[N,n,d]_q$-code such that $d\geq n$. 
In particular, if $N_q(n)$ is the minimum length of a linear $[N,n,n]_q$-code then the 
tensor rank $\mu_q(n)$ of multiplication in the finite 
field $\F_{q^n}$ is such that $\mu_q(n)\geq N_q(n)$.    
\end{coro}

Let us recall that there exists a continuous decreasing function $\alpha^\mathrm{lin}_q(\delta)$  
on the segment $[0, 1-\frac{1}{q}]$ 
which corresponds to the bound for the rate $R$  of the linear codes over $\F_q$ with relative minimum distance at least 
$\delta$ (cf \cite[1.3.1]{tsvl2}). Hence:

\begin{coro} \label{deltacode}
One has: $$m_q\geq \delta^{-1}_q,$$ where $\delta_q$ is the unique solution of the equation 
$\alpha^\mathrm{lin}_q(\delta)=\delta$.
\end{coro}

Any upper bound for $\alpha^\mathrm{lin}_q(\delta)$ gives an upper bound for  $\delta_q$ and thus a lower bound for $m_q$.
So, from this corollary, it follows that we can obtain lower bounds of the asymptotic quantity $m_q$ from asymptotic 
parameters of codes. Now, let us summarize the known lower bounds concerning this quantity, 
namely the lower bound of $m_2$ obtained by R. Brockett, M. Brown and D. Dobkin in \cite{brdo1,brdo2} 
by using the bound of ``four'' \cite[1.3.2]{tsvl2} 
for asymptotic parameters of binary codes, and the lower bound of $m_q$ for $q>2$ given by 
I. Shparlinski, M. Tsfasman and S. Vladut in \cite{shtsvl} 
by using the asymptotic Plotkin bound \cite[1.3.2]{tsvl2}. 
Note that this last bound is a straightforward consequence of Proposition~4.3 established 
by D.V. and G.V. Chudnovsky  in \cite{chch}.

\begin{propo}
One has:
$$m_2\geq 3.52$$ and $$m_q\geq 2\left( 1+\frac{1}{q-1}\right) \mbox{ for any }q>2.$$

\end{propo}


\subsection{Supercodes} 

Let us recall the notion of supercode introduced by Shparlinski, Tsfasman and Vladut in \cite{shtsvl}. 
First, let us recall the idea leading 
to the emergence of the notion of supercode. By Section \ref{connectcode}, any decomposition of 
the tensor $t_M$ into a sum of $N$ 
summands of rank one enables us to obtain an $[N,n,d]_q$-code. In fact, the notion of  
supercode follows from the question to know when 
it is possible conversely to construct such a decomposition from  a linear $[N,n,\geq n]_q$-code.

\begin{defi}\label{defsupercode}
Let $S \subseteq \F_{q^n} \oplus \F_q^{N}$ be an $\F_q$-linear subspace.
$S$ is called an $[N,n]_q$-supercode if the following conditions are satisfied:
\begin{enumerate}[1)]
	\item the first projection $$\pi_1: \F_{q^n} \oplus \F_q^N \longrightarrow \F_{q^n}$$ restricted to $S$ is surjective.
	\item let ${S^2=\{s_1s_2\, | \, s_1, s_2 \in S\}}$ where the multiplication is 
	that in $\F_q$-algebra \linebreak[4]${\F_{q^n} \oplus \F_q^{N}}$ 
	and let $<S^2>$ be the subspace in ${\F_{q^n} \oplus \F_q^{N}}$ spanned by $S^2$. 
	The second projection $$\pi_2: \F_{q^n} \oplus \F_q^N \longrightarrow \F_q^N$$ restricted to $<S^2>$ is injective.
\end{enumerate}
\end{defi}

From Definition \ref{defsupercode}, it is now possible to obtain the following more restrictive notion, almost equivalent 
to the notion of symmetric decomposition of a multiplication tensor.

\begin{defi}\label{defsupercodeexact}
An $[N,n]_q$-supercode $S$ is said exact if $\pi_1$ is an isomorphism, i.e. if ${dim\,S=n}$.
\end{defi}

\begin{propo} \label{proprisuper}
Let S be an $[N,n]_q$-supercode and let $C=\pi_2(S)$, then:
\begin{enumerate}[(1)]
	\item $C$ is an $[N,\geq n, \geq n]$-code.
	\item If S is exact then $C$ is an $[N, n, \geq n]$-code.
	\item Any supercode contains an exact sub-supercode.
\end{enumerate}
\end{propo}

In fact, the notion of exact supercode is equivalent to that of symmetric decomposition of $t_M$ 
into a sum of $N$ rank one tensors, up to the representation of $\F_{q^n}$ 
(i.e modulo the following equivalence relation):

\begin{defi}\label{tensorequival}
Let ${\sigma_1=\sum_{i=1}^N u_i\otimes u_i\otimes w_i}$ and ${\sigma_2=\sum_{i=1}^N v_i\otimes v_i\otimes z_i}$ 
be two symmetric decompositions of $t_M$.
We call $\sigma_1$ and $\sigma_2$ equivalent if $u_i=v_i$ for every~$i$.
\end{defi}

Now, by considering the equivalence relation of Definition \ref{tensorequival}, we obtain the following result. 

\begin{theo}
There is a bijection between the set of exact supercodes and the set of equivalence classes of 
symmetric decompositions of $t_M$. 
\end{theo}

Then, by \cite[Proposition 1.11 and Corollary 1.13]{shtsvl}, we obtain:

\begin{coro}\label{supercodealgomult}
\begin{enumerate}[(1)]
        
	\item Any exact supercode ${S\subset \F_{q^n}\oplus \F_q^N}$ yields a symmetric 
	multiplication algorithm of bilinear complexity $N$ 
and conversely.
	\item Any supercode ${S\subset \F_{q^n}\oplus \F_q^N}$ yields a symmetric multiplication 
	algorithm of bilinear complexity $\leq N$.
\end{enumerate}
\end{coro}

 Note that I. Shparlinski, M. Tsfasman and S. Vladut in \cite{shtsvl} gave an explicit construction 
 of a symmetric tensor $t_M$ of length $N$ performing the multiplication in a finite field $\F_{q^n}$ 
 from an exact supercode ${S\subset \F_{q^n}\oplus \F_q^N}$. Conversely, from an arbitrary symmetric 
 decomposition, they explicitly obtain an exact supercode by \cite[Proposition 1.11]{shtsvl}.
 
 \begin{rema}
 Note that certain authors use the notion of multiplication friendly code which is equivalent 
 to the notion of exact supercode. 
 In particular, the results obtained by using the notion of multiplication friendly code only 
 concern the symmetric bilinear complexity.
 \end{rema}
 
\begin{op}
How can one characterize those $[N,\geq k,\geq k]$-codes which are projections of supercodes?
\end{op}

\section{Generalizations of the algorithm of Chudnovsky-Chudnovsky}\label{SecGeneralizations}

\subsection{Motivation}\label{motiv}

When using the original Chudnovsky-Chudnovsky method,
one sees that the bounds that can be obtained on the bilinear complexity, as well as their effectivity
or the practical implementation of the corresponding multiplication algorithms,
highly depend on the choice of the geometric data on which Theorem~\ref{chudchud} is applied.
For instance, in order to get the best possible bounds, one needs curves having sufficiently many rational points with the smallest possible genus.
This works well when one is considering a base field that is not too small, and of square order, so the celebrated Drinfeld-Vladut bound can be attained (see section~\ref{CurveChoice} for details).
But in other situations, the original Chudnovsky-Chudnovsky method presents certain limitations.
Several improvements were then proposed to overcome these limitations.

In order to better understand these improvements, we will thus distinguish \emph{two steps} in the construction of multiplication algorithms.
The \emph{first step} is to state a ``generic'' CCMA, which takes as input some geometric data
(a function field or a curve, some places or points on it, and some divisors that satisfy adequate conditions),
and gives as output an effective multiplication algorithm, or at least an upper bound on some bilinear complexity.
The \emph{second step} then is to specify the geometric objects on which this generic CCMA will be 
applied: choice of the curves,
existence of the divisors, etc.

Concerning the \emph{first step} (generic statement of the CCMA), successive generalizations were proposed by various authors,
using several independent ingredients, among which we can cite:
\begin{itemize}
\item evaluation at places of higher degree and/or with multiplicities
\item symmetric/asymmetric versions of the algorithm optimized for symmetric/asymmetric bilinear complexity respectively
\item formulation adapted for an iterative use.
\end{itemize}

In this section we give more details on these lines of improvements, with emphasis on the first two (in sections \ref{eval_deg_mult} and \ref{DiscussionSym}),
and we present the best finalized version of the CCMA \cite[Theorem 3.5]{randJComp}, which combines them all.
We then explain how intermediate historical contributions can be retrieved as particular cases.

Concerning the \emph{second step} (specification of the geometric objects), the most important ingredients are:
\begin{itemize}
\item careful choice of the curves, either explicit recursive towers, 
their densification and descent of base field (see section~\ref{descent} for details), or more abstract modular, Shimura, or Drinfled modular curves (see section~\ref{courbes_Shimura})
\item techniques to ensure the existence, or even to effectively construct the divisor needed 
to perform interpolation, of best possible degree; this is especially important in the context of symmetric algorithms (see section~\ref{sym_methods}).
\end{itemize}

Of course these two steps that we distinguished are closely intertwined: a suitably generalized generic CCMA will allow
a broader choice for the geometric objects, hence lead to better bounds or a more effective implementation.
In the other direction, it can happen that some geometric conditions
(e.g. existence of points of given degree or of suitable divisors)
can be replaced with simple numerical criteria, and get included in the statement of the generic CCMA.

\subsection{Evaluation at places of higher degree and with multiplicities}
\label{eval_deg_mult}

Here one can cite several successive contributions.
\begin{itemize}
\item First S. Ballet and R. Rolland have generalized in \cite{baro1} the algorithm using places of degree $1$ and~$2$.
\item Then N. Arnaud \cite{arna1} introduced, as in the interpolation of Lagrange-Sylvester,
the use of derivatives (evaluation with multiplicities) to improve the interpolation process.
\item These ideas are combined and extended in the work of M. Cenk and F. \"Ozbudak in \cite{ceoz}. 
This generalization uses several coefficients in the local expansion at each place $P_i$ 
instead of just the first one. Due to the way it is obtained, their bound for the bilinear complexity
involves a sum of local contributions, each of which is written as a product of two separate factors:
one factor accounts for the degree of the place, the other factor accounts for the multiplicity. 
\item Last H.~Randriambololona \cite{randJComp} refined this method by introducing a single quantity
that combines both degree and multiplicity at the same time and leads to the sharpest bounds as presently known.
\end{itemize}
This quantity introduced in \cite{randJComp} can be defined in two variants,
one for the bilinear complexity, the other for the symmetric bilinear complexity:
\begin{defi}
For any integers $m,\ell\geq 1$ we consider the $\F_q$-algebra $\F_{q^m}[t]/(t^\ell)$ of polynomials 
in one indeterminate with coefficients in $\F_{q^m}$, truncated at order $\ell$,
and we denote by 
$$\mu_q(m,\ell) = \mu((\F_{q^m}[t]/(t^\ell))/\F_q)$$
its bilinear complexity over $\F_q$, and by 
$$\mus_q(m,\ell) = \mus((\F_{q^m}[t]/(t^\ell))/\F_q)$$
its symmetric bilinear complexity over $\F_q$. 
\end{defi}

Note that for $\ell=1$, we have $\mu_q(m,1)=\mu_q(m)$ and $\mus_q(m,1)=\mus_q(m)$.
While for $m=1$, we have $\mu_q(1,\ell)=\widehat{M}_q(\ell)$ as defined by M. Cenk and F. \"Ozbudak in \cite{ceoz}
(we could set likewise $\mus_q(1,\ell)=\widehat{M}_q^{\mathrm{sym}}(\ell)$, although this quantity is not considered in \cite{ceoz}).

\vspace{\baselineskip}

The generalized evaluation maps that appear in the generalized CCMA can be described either in the language of modern algebraic geometry,
as done in \cite{randJComp}, or in the language of algebraic function fields, as done in previous works. 
Actually these two languages are equivalent, so we explain how to pass from one to the other.

Suppose we are given:
\begin{itemize}
\item a curve $X$ over $\F_q$ (which corresponds to a function field $F/\F_q$)
\item a closed point $P$ on $X$ of degree $m$ (which corresponds to a place of $F$ of degree $m$)
\item an integer $\ell$.
\end{itemize}
This allows to consider the thickened point $P^{[\ell]}$ on $X$, which is the closed subscheme defined by the sheaf of ideals $(\mathcal{I}_P)^\ell$.

Now, for any divisor $\D$ on $X$, we can define a generalized evaluation map, that evaluates sections of $\mathcal{D}$ at $P$
with multiplicity $\ell$. In geometric terms, this is just the natural restriction map
$$\varphi_{\D,P,\ell}:\Ld{}\longrightarrow\mathcal\mathcal{O}_X(\D)|_{P^{[\ell]}}.$$
After possibly replacing $\D$ with a linearly equivalent divisor, we will assume $P$ is not in the support of $\D$.
We then have a natural identification
$\mathcal{O}_X(\D)|_{P^{[\ell]}}=\mathcal{O}_{P^{[\ell]}}$.
Then, thanks to \cite[Lemma 3.4]{randJComp}, we have an isomorphism of algebras
$$\mathcal{O}_{P^{[\ell]}}\simeq\F_{q^m}[t]/(t^\ell)$$
where $t$ corresponds to a local parameter $t_P$ at $P$, and $\F_{q^m}$ is identified with the residue field of $P$.
Last, in order to make everything explicit for computations, we can use the natural linear isomorphism

$\F_{q^m}[t]/(t^\ell)\simeq(\F_{q^m})^\ell$
identifying a polynomial $a_0+a_1t+\cdots+a_{\ell-1}t^{\ell-1}$ with its coefficients $(a_0,a_1,\dots,a_{\ell-1})$.
Combining all this, the generalized evaluation map becomes
\begin{equation}
\label{gen_eval}
\varphi_{\D,P,\ell}:\left|\begin{array}{ccl}
\Ld{} & \longrightarrow & (\F_{q^m})^\ell\\
f & \mapsto & (f(P),f'(P),\dots,f^{(\ell-1)}(P))
\end{array}\right.
\end{equation}
where the $f^{(k)}(P)$ are the coefficients of the local expansion
$$f=f(P)+f'(P)t_P+f''(P)t_P^2+\cdots+f^{(k)}(P)t_P^k+\cdots$$
of $f$ at $P$ with respect to $t_P$. Sometimes this is also called a ``derived evaluation map'',
although one should be careful that for $k\geq2$ these $f^{(k)}(P)$ are not precisely derivatives in the usual sense
(at best they are ``$\frac{1}{k!}$ times the derivative'').

\subsection{Discussion on symmetry}\label{DiscussionSym}

In the broader context of bilinear algorithms over finite fields, the distinction between
(general) bilinear complexity and symmetric bilinear complexity, together with some of the
mathematical issues related specifically to the construction of symmetric algorithms, were first discussed
in 1984 by Seroussi and Lempel with~\cite{sele}.

Focusing now on works based on the Chudnovsky-Chudnovsky method, it turns out that until 2011, all results
(including those in \cite{chch}\cite{shtsvl}\cite{baro2}\cite{ceoz}\cite{rand2D-G}) were stated in terms
of $\mu_q$ only (not $\mus_q$), although by construction the method always produced symmetric algorithms.
Of course this does not mean that the authors were not aware of the distinction:
indeed, for instance, I. Shparlinski, M. Tsfasman and S. Vladut explicitely mentioned the issue
when they observed \cite[p.~154]{shtsvl} that their notion of supercode corresponds only to symmetric algorithms.

However the situation became unsatisfactory when I. Cascudo, R. Cramer and C. Xing discovered 
the gap in the construction
of the divisor in~\cite{shtsvl}, as already discussed in section~\ref{SectionAlgCurvApproach}. Indeed, it turns
out that the difficulty of this construction, which they analyze in terms of the $2$-torsion in the
divisor class group of the curve (see section~\ref{Bound2torsion}),
is closely related to the symmetry requirement for the algorithm.

Finally, things were clarified by H. Randriambololona in \cite{randJComp}.
Along with the contributions already discussed in section~\ref{eval_deg_mult}, this work
introduced two further improvements to the method:
\begin{itemize}
\item one that solves the difficulty with the construction of the divisor in the symmetric case, at least 
for curves with sufficiently many rational points (see section~\ref{construction_sym} for details)
\item another one that produces asymmetric algorithms instead, by allowing asymmetry in the CCMA;
this is advantageous because asymmetric interpolation allows more freedom in the choice of the divisors,
and ultimately, can lead to sharper bounds.
\end{itemize}

As a consequence of these developments, whenever possible, the generalized CCMA should be stated
in two versions, one for bilinear complexity, the other for symmetric bilinear complexity.
Likewise, the numerical bounds should be stated in two versions, accordingly.

\vspace{\baselineskip}

Beside bilinear complexity $\mu_q$ and symmetric bilinear complexity $\mus_q$,
other refinements were introduced and studied in \cite{sele} and \cite[Appendix~A]{randAGCT}:
these are trisymmetric bilinear complexity $\mu^\mathrm{tri}_q$, and normalized trisymmetric 
bilinear complexity $\mu^\mathrm{nrm}_q$.

It should be noted that it can happen that these quantities are not well defined for some values of $q$ and $n$.
More precisely, \cite[Prop.~A.14]{randAGCT} shows that $\mu^\mathrm{tri}_q(n)$ is well defined for 
all values of $q$ and $n$ except precisely for $q=2,\,n\geq3$.
Likewise \cite[Prop.~A.19]{randAGCT} shows that $\mu^\mathrm{nrm}_q(n)$ is well defined for 
all values of $q$ and $n$ except precisely for $q=2,\,n\geq3$ and for $q=4,\,n\geq2$.

In any case, when well defined, one has
$$\mu_q(n)\leq\mus_q(n)\leq\mu^\mathrm{tri}_q(n)\leq\mu^\mathrm{nrm}_q(n).$$
Also, \cite[Th.~2]{sele} gives
$\mu_q^{\mathrm{tri}}(n)\leq4\,\mu_q^{\mathrm{sym}}(n)$ for $q\neq2$, $\mathrm{char}(\F_q)\neq3$,
and \cite[Prop.~A.19]{randAGCT}
gives
$\mu_q^{\mathrm{nrm}}(n)\leq 2\,\mu_q^{\mathrm{tri}}(n)$ for $q\neq7$
and
$\mu_7^{\mathrm{nrm}}(n)\leq 3\,\mu_7^{\mathrm{tri}}(n)$.
Joint with the linearity of $\mus_q$, this gives the linearity
of $\mu_q^{\mathrm{nrm}}$ and $\mu_q^{\mathrm{tri}}$ for most $q$.

But beside this, very few is known about these quantities.

\begin{op}
What are the exact values of $\mu_q^{\mathrm{nrm}}(n)$ and $\mu_q^{\mathrm{tri}}(n)$ for small $q$ and $n$?

Can some of the inequalities between $\mu_q(n)$, $\mus_q(n)$, $\mu^\mathrm{tri}_q(n)$ and $\mu^\mathrm{nrm}_q(n)$ be strict? If so, for which values of $n$?

Can one give better asymptotic bounds on them?
\end{op}

\vspace{\baselineskip}

\subsection{The current generalized CCMA}\label{CurrentCCMA}

Now we can state H. Randriambolona 's result \cite[Theorem 3.5]{randJComp}, which 
provides the current most general CCMA.
It makes use of the most elaborate form of derived evaluation, and it gives bounds both 
for asymmetric complexity and for symmetric complexity.

As already explained, this result was originally presented in the language of modern algebraic geometry,
but here we give the equivalent translation in the language of function fields.

\begin{theo} \label{theo_evalder}
Let
\begin{itemize}
	\item $q$ be a prime power,
	\item $F/\F_q$ be an algebraic function field,
	\item $Q$ be a place of $F/\F_q$, of degree $n=\deg Q$
        \item $\ell$ be a positive integer
	\item $\D_1,\D_2$ be two divisors of $F/\F_q$,
	\item ${\mathscr P}=\{P_1,\ldots,P_N\}$ be a set of places of arbitrary degree $d_i=\deg P_i$,
	\item $u_1,\ldots,u_N$ be positive integers.
\end{itemize}
We suppose that $Q$ and all the places in $\mathscr P$ are not in the support of $\D_1$ and $\D_2$, and that:
\begin{enumerate}[(a)]
	\item the maps
	$$\varphi_{\D_1,Q,\ell}:\mathcal{L}(\D_1)\longrightarrow (\F_{q^n})^\ell$$
        and
 	$$\varphi_{\D_2,Q,\ell}:\mathcal{L}(\D_2)\longrightarrow (\F_{q^n})^\ell$$
 	are onto,
	\item the map
	$$Ev_{\mathscr P,\underline{u}} :  \left |
	\begin{array}{ccl}
	\mathcal{L}(\D_1+\D_2) & \longrightarrow & \left(\F_{q^{d_1}}\right)^{u_1} \times 
	\left(\F_{q^{d_2}}\right)^{u_2} \times \cdots \times \left(\F_{q^{\deg d_N}}\right)^{u_N} \\
	f & \longmapsto & \big(\varphi_1(f), \varphi_2(f), \ldots, \varphi_N(f)\big)
	\end{array} \right.$$
	is injective,
\end{enumerate}
where the applications $\varphi_{\D_1,P,\ell}$, $\varphi_{\D_2,P,\ell}$, and $\varphi_i=\varphi_{\D_1+\D_2,P_i,u_i}$ 
are the derived evaluation maps from \eqref{gen_eval}.
Then 
$$\mu_q(n,\ell) \leq \displaystyle \sum_{i=1}^N \mu_q(d_i,u_i).$$
Moreover, if $\D_1=\D_2$, the same holds for the symmetric bilinear complexity:
$$\mus_q(n,\ell) \leq \displaystyle \sum_{i=1}^N \mus_q(d_i,u_i).$$

Existence of the objects satisfying the conditions above is ensured by the following numerical criteria:
\begin{itemize}
\item a sufficient condition for the existence of $Q$ of degree $n$ is that $2g+1\leq q^{(n-1)/2}(q^{1/2}-1)$, 
where $g$ is the genus of $F$
\item a sufficient condition for (a) is that the divisors $\D_1-\ell Q$ and $\D_2-\ell Q$ are nonspecial: 
$$i(\D_1-\ell Q)=i(\D_2-\ell Q)=0$$
where $i$ denotes index of speciality
\item a necessary and sufficient condition for (b) is that the divisor $\D_1+\D_2-\mathcal{G}$ is zero-dimensional:
$$\dim\mathcal{L}(\D_1+\D_2-\mathcal{G})=0$$
where $\mathcal{G}=u_1P_1+\cdots+u_NP_N$.
\end{itemize}
\end{theo}

The fact that $\mu_q(n,\ell)$ (resp. $\mus_q(n,\ell)$) appears on the left-hand side of the 
inequalities allows to apply the result recursively.
For $n=1$ it also provides bounds for the quantity $\widehat{M}_q(\ell)$ of M. Cenk and F. \"Ozbudak
(resp. for $\widehat{M}_q^{\mathrm{sym}}(\ell)$).

However in most applications we are interested mostly in the case $\ell=1$.
If we restate the result in this particular case, and focus only on the symmetric part, 
this generalized version of CCMA algorithm then specializes to the following 
statement (special case of \cite[Theorem 3.5]{randJComp}), which suffices for most applications:

\begin{coro}\label{theo_evalder_simple}
Let
\begin{itemize}
	\item $q$ be a prime power,
	\item $F/\F_q$ be an algebraic function field,
	\item $Q$ be a place of $F/\F_q$, of degree $n=\deg Q$ and residue field $F_Q\simeq\F_{q^n}$
 	\item $\D$ be a divisor of $F/\F_q$,
	\item ${\mathscr P}=\{P_1,\ldots,P_N\}$ be a set of places of arbitrary degree $d_i=\deg P_i$,
	\item $u_1,\ldots,u_N$ be positive integers.
\end{itemize}
We suppose that $Q$ and all the places in $\mathscr P$ are not in the support of $\D$, and that:
\begin{enumerate}[(a)]
	\item the evaluation map
	$$\varphi_{\D,Q}:\left|\begin{array}{ccl}\mathcal{L}(\D)&\longrightarrow&\F_{q^n}\\f&\mapsto&f(Q)\end{array}\right.$$
	is onto
        \item the map
	$$Ev_{\mathscr P,\underline{u}} :  \left |
	\begin{array}{ccl}
	\mathcal{L}(2\D) & \longrightarrow & \left(\F_{q^{d_1}}\right)^{u_1} \times \left(\F_{q^{d_2}}\right)^{u_2} \times \cdots \times \left(\F_{q^{\deg d_N}}\right)^{u_N} \\
	f & \longmapsto & \big(\varphi_1(f), \varphi_2(f), \ldots, \varphi_N(f)\big)
	\end{array} \right.$$
	is injective, where $\varphi_i=\varphi_{2\D,P_i,u_i}$ is the derived evaluation map from \eqref{gen_eval}.
\end{enumerate}
Then 
$$\mus_q(n) \leq \displaystyle \sum_{i=1}^N \mus_q(d_i,u_i).$$
\end{coro}

\vspace{\baselineskip}

This can be specialized still further.
Indeed, first observe that for all $d,u$ we have the easy inequality
\begin{equation}\label{IneqFormRecur}
\mus_q(d,u)\leq\mus_q(d)\widehat{M}_{q^d}^{\mathrm{sym}}(u).
\end{equation}
This follows directly from Lemma~\ref{functorial_inequalities}(b) applied with 
$\mathcal{A}=\F_{q^d}[t]/(t^u)$, $L=\F_{q^d}$, $K=\F_q$.
We deduce:

\begin{coro}\label{ceoz_sym}
Under the same hypotheses as Corollary \ref{theo_evalder_simple}, we have
$$\mus_q(n) \leq \displaystyle \sum_{i=1}^N\mus_q(d_i)\widehat{M}_{q^{d_i}}^{\mathrm{sym}}(u_i).$$
\end{coro}

Corollary \ref{ceoz_sym} can be seen as a symmetric variant of M. Cenk and F. \"Ozbudak's version of the CCMA \cite{ceoz}.
It is weaker than Corollary \ref{theo_evalder_simple},
since the inequality $\mus_q(d,u)\leq\mus_q(d)\widehat{M}_{q^d}^{\mathrm{sym}}(u)$ can be strict.

One should be careful that all bilinear complexities in the original statement of \cite{ceoz} (including the one for multiplicities)
have to be replaced by symmetric bilinear complexities in order to get this valid symmetric reformulation.

Going further back in time,
let us then remark that the algorithm given in \cite{chch} by D.V. and G.V. Chudnovsky corresponds to the 
case $d_i=1$ and $u_i=1$ for $i=1, \ldots, N$. 
The first generalization introduced by S. Ballet and R. Rolland in \cite{baro1} 
concerns the case $d_i=1 \mbox{ or }2$ and $u_i=1$ for $i=1, \ldots, N$. 
Next, the generalization introduced by N. Arnaud in \cite{arna1} concerns 
the case $d_i=1 \mbox{ or }2$ and $u_i=1\mbox{ or }2$  for $i=1, \ldots, N$. 
In particular, as a corollary of Theorem \ref{theo_evalder}, 
we have the following result 
obtained by N. Arnaud in \cite{arna1} by gathering the places 
used with the same multiplicity; namely he sets 
${\ell_j:= | \{ P_i\, | \, \deg P_i =j \mbox{ and } u_i=2 \} |}$ for ${j=1}$ and~$2$ and with $\D=\D_1=\D_2$.

\begin{coro} \label{theo_deg12evalder}
Let 
\begin{itemize}
	\item $q$ be a prime power,
	\item $F/\F_q$ be an algebraic function field,
	\item $Q$ be a degree $n$ place of $F/\F_q$, 
	\item $\D$ be a divisor of $F/\F_q$,
	\item ${\mathscr P}=\{P_1,\ldots,P_{N_1},P_{N_1+1},\ldots,P_{N_1+N_2}\}$ be a set of 
$N_1$ places of degree\\ \indent one and $N_2$ places of degree two,
	\item ${0 \leq \ell_1 \leq N_1}$ and ${0 \leq \ell_2 \leq N_2}$ be two integers.
\end{itemize}
We suppose that $Q$ and all the places in $\mathscr P$ are not in the support of $\D$ and that:
\begin{enumerate}[(a)]
	\item the map
	$$
	Ev_Q: \Ld{} \rightarrow \F_{q^n}\simeq F_Q$$
	is onto,
	\item the map
	 $$
	 Ev_{\mathscr P}: \left |
	\begin{array}{ccl}
   	\Ld{2} & \rightarrow & \F_{q}^{N_1} \times \F_{q}^{\ell_1}\times \F_{q^2}^{N_2} \times \F_{q^2}^{\ell_2}  \\
    	f & \mapsto & \big(f(P_1),\ldots,f(P_{N_1}),f'(P_1),\ldots,f'(P_{\ell_1}),\\
	  &  & \ f(P_{N_1+1}),\ldots,f(P_{N_1+N_2}),f'(P_{N_1+1}),\ldots,f'(P_{N_1+\ell_2})\big)
	\end{array} \right .
	$$
is injective.
\end{enumerate}
Then  
$$
\mus_q(n)\leq  N_1 + 2\ell_1 + 3N_2 + 6\ell_2.
$$
\end{coro}

I\section{Choice of the curves}\label{CurveChoice}

\subsection{Motivation and notations}

As seen in Section \ref{SectionAlgCurvApproach} and \ref{SecGeneralizations}, until now, the best method to quantify the bilinear 
complexity of multiplication in finite fields is the CCMA algorithm based upon the interpolation over algebraic 
curves defined over a finite field. So in this context, to get the best bounds on the upper-limit complexities $M_q$ and $\Ms_q$ or the upper bounds $C_q$ and $C^{sym}_q$ defined in Section \ref{mM}, it is necessary to use sufficiently many different curves so as to deal with the worst cases. 
So let us give a name to the following requirement, formalized in \cite[Claim p163]{shtsvl}:

\begin{defi} \label{dense} 
Let $X_s/k$ be a family of curves over a field $k$ with genera $g_s$. 
We say that the family $\left(X_s\right)_s$ is {\it dense} if and only if the genera $g_s$ tend to 
infinity and the ratio of two successive genera $g_{s+1}/g_s$ tends to 1. 
\end{defi}

As introduced in the last section, multiplication algorithms by interpolation on 
algebraic curves often require many points of higher degree $r\geq 2$. So let us study the 
best possible asymptotic ratios $\beta_r$ of the number of places of degree $r$ divided by the genus. The first definition is due to M. Tsfasman \cite{tsfa} (cf. also \cite[definitions 1.1, 1.2 and 1.3]{baro4}).

\begin{defi}
Let ${\mathcal X}/\F_q=(X_s/\F_q)$ be a sequence of curves $X_s/\F_q$ defined over a 
finite field $\F_q$ of genus $g_s=g(X_s/\F_q)$. 
We suppose that the sequence of the genus $g_s$ is an increasing sequence growing to infinity. 
Then the sequence ${\mathcal X}/\F_q$ 
is said to be asymptotically exact if for all $m\geq 1$ the limit 
$\beta_r({\mathcal X})=\lim_{s\rightarrow\infty}\frac{B_r(X_s)}{g_s}$, where $B_r(X_s)$ 
denotes the number of closed points of degree $r$ of the curve $X_s$, exists. 
\end{defi}

\begin{defi}\label{DefIharaConstant}
Let $r\geq 1$ be an integer and $q$ a prime power. For $X$ a curve over $\F_q$, let $B_r(X)$ 
denote the number of closed points of degree $r$.
For an asymptotically exact sequence of curves ${\mathcal X}=(X_s)$, let us 
define $$\beta_r({\mathcal X})=\lim_{s\rightarrow\infty}\frac{B_r(X_s)}{g_s}.$$
Then, we respectively define : 
$$\betar_r(q)(\hbox{resp. } \betad_r(q)) =\limsup_{{\mathcal X}}\beta_r({\mathcal X}),$$ 
${\mathcal X}$ running over all asymptotically exact sequences of curves 
(resp. dense asymptotically exact sequences of curves).
\end{defi}

\begin{rema}
Note that the quantity $A_1(q)$ is the classical Ihara Constant $A(q)$ defined by Y. Ihara in \cite{ihar}.
The order $r$ Ihara constants $A_r(q)$ were in particular defined in \cite[definitions 1.3]{baro4}.
Concerning the quantities $\betad_r(q)$, note that the dense Ihara constant $\betad_1(q)$ was first introduced (and noted $A'(q)$)
by H. Randriambololona in \cite{rand2D-G} (cf. also \cite{randGap2}). The order $r$ dense Ihara constants $\betad_r(q)$ 
were first introduced (and noted $\betadRamb_r(q)$) by M. Rambaud in \cite{ramb1}.
\end{rema}

The following is possibly well-known. It essentially follows from \cite[Lemma IV.3]{cacrxiya}, itself 
based on the generalized bound of Drinfeld-Vladuts (cf. \cite[Theorem 1]{tsfa}, see also \cite[Definitions 1.2 and 1.3]{baro4}).

\begin{theo} \label{th:dv}
Let $\left(X_s/\F_q\right)$ be a family of curves over a finite field $\F_q$, with genera $g_s$ tending 
to infinity. Let $r\geq 1$ be an integer, $B_r(X_s)$ the number of closed points of 
degree $r$ and $\mid X_s(\F_{q^r}) \mid$ the number of points of $X_s$ in the extension 
$\F_{q^r}$. Then the following assertions are equivalent :
\begin{align}\tag{i} \lim_{s\rightarrow\infty}\frac{\left|X_s(\F_{q^r})\right |}{g_s} & = \sqrt{q^r}-1, \\ 
\tag{ii} \label{DVordreR}  \lim_{s\rightarrow\infty}\frac{B_r(X_s)}{g_s} & = \frac{\sqrt{q^r}-1}{r}.  
\end{align}
\end{theo}

As a corollary of Theorem 1 in \cite{tsfa}, the following holds:

\begin{theo}
\begin{equation}\label{Atilder}
\betad_r(q) \leq \betar_r(q) \leq \frac{\sqrt{q^r}-1}{r}.
\end{equation}
\end{theo}

\subsection{Explicit towers, densification and descent}\label{descent}

The pioneer papers \cite{chch} \cite{shtsvl} having for objectives to prove the linearity (cf. Section \ref{mM}) 
of this complexity with respect to the extension degree, required the use of 
infinite families of curves with many rational points relatively to the genus. 
However, the first exhibited families of curves (of type modular and Shimura) 
enable them to obtain uniquely purely asymptotic bounds. So, the objective of \cite{ball0} (cf. also \cite{ball1} 
and footnote \ref{borunivTsfa} page \pageref{borunivTsfa}) was to give 
the first uniform upper bounds with respect to $q$. In this aim, it was necessary to use more explicit families of curves. 
The first tower of algebraic function fields of Garcia-Stichtenoth \cite{gast} 
fulfilled the required conditions: knowledge of fundamental invariants, 
namely the genus and the number of rational points of each step of the tower, 
which attains the Drinfeld-Vladut bound. From a general point of view, to obtain the 
best bounds by CCMA, we need to use families of curves of genus increasing the more slowly 
possible (cf. Section \ref{motiv} and Theorem \ref{theoprinc} in Section \ref{mus}). But, a tower 
of algebraic function fields is composed of successive algebraic function fields whose genera increase 
as the extension degree between two consecutive steps by the Hurwitz formula.
For example, the first Garcia-Stichtenoth tower defined over $\F_{q^2}$ 
is an Artin-Schreier tower whose ratio of two consecutive genus is $\frac{g_{i+1}}{g_i}\geq q$ where $q$ is an arbitrary prime power.
In this case, an interesting strategy to improve the bounds obtained with this type of tower consisted on densifying this tower 
by adding intermediate steps (cf. \cite{ball2}). It is easily possible in this case, even without knowing the recursive equation of 
intermediate steps because the tower is a Galois tower.
When the used towers ${\mathcal X}/\F_q$ are such that the value of $\beta_1({\mathcal X})$ is not sufficiently large
(which is the case when the finite fields of definition are small or when the best known lower bound of 
the Ihara constant $\betar_r(q)$ associated to the definition field $\F_q$ is not sufficiently large), 
it is necessary to use places of degree $>1$ because of the Drinfeld-Vladut bound (cf. \cite{baro1}, \cite{bapi}). 
So, we need families of curves reaching the Drinfeld-Vladut Bound of order $r>1$ 
(cf. \cite{baro2} and Assertion (\ref{DVordreR}) in Theorem \ref{th:dv}). Until now, the only way to obtain such families is the technic of the descent of families of algebraic function fields defined over $\F_{q^r}$ on the definition field $\F_{q}$, which was introduced in \cite{baro1}. Of course, the descent of the original tower of Garcia-Stichtenoth is always possible 
since the coefficients of the recursive equation lie in $\F_q$. However, the problem arises as soon as we introduce intermediate steps. 
So, in \cite{baro1}, the descent was made explicit only for the characteristic two and $r=2$ because 
in this case the descended tower conserves the property to be Galois. Then, the generalization for any characteristic with $r=2$ 
was realized in \cite{balbro} by using two different techniques: theoretically by using the action of the Galois group of $\F_{q^2}/\F_q$ 
on the intermediate steps of the tower defined on $\F_{q^2}$ or by finding explicit equations of the 
intermediate steps. Then, having used all the possibilities of the towers, it became necessary to use 
families of algebraic function fields more dense than the towers. In this aim, 
it was natural to come back to the study of families of modular and Shimura curves, which is the subject of the following section.


\subsection{Modular and Shimura curves}\label{courbes_Shimura}

The previous section motivates the search for dense families of curves becoming optimal after a base field extension of (small) degree $r$.

\medskip
Firstly, the towers of Garcia-Stichtenoth \cite{gast}\cite{gast2} being actually defined over 
their prime field $\F_p$, then for any base extension degree $r$, there exists \emph{non-dense} 
towers reaching the previous bound (see next section):

\begin{equation}\label{DVR}
\betar_r(q)=\frac{\sqrt{q^r}-1}{r} \text{ as long as $q^r$ is a square.}
\end{equation}

Now, in the particular case of \emph{quadratic extensions} $r=2$, the celebrated results of 
\cite{ihar} and \cite{tvz} (cf. also \cite{shtsvl}) state that (see also the two original 
approaches of \cite[Theorem IV.4.5]{duce}), for all prime power $q$, there exists \emph{dense} families of 
\emph{Shimura modular curves} over $\F_q$ that become optimal over $\F_{q^2}$. See also 
\cite{voi} for an introduction (in characteristic zero).
Notice that classical modular curves over prime fields $\F_p$ are a particular case of Shimura curves. 
Summing up, the Shimura curves mentionned above match the bound of Drinfeld-Vladuts over $\F_{q^2}$, which reads:

\begin{equation}\label{DV}
\betad_1(q^2) = q-1. 
\end{equation}

Plus, taking into consideration that these curves are defined over $\F_q$, Theorem \ref{th:dv} implies :
\begin{equation}\label{DV2}
\betad_2(q) = \frac{q-1}{2}.
\end{equation}


\subsubsection{Intertwinning two recursive towers into a dense family}
A recursive construction to obtain a dense family of curves consists in \emph{intertwinning 
two towers of modular curves defined over the same basis}. Let us illustrate this with the 
classical modular curves $X_0(N)$. Let $l$ be a prime number, then we know from Igusa 
that there exists ---canonical--- models $X_0(l^i)_\Q$ over $\Q$ for any $i\geq 0$, which have 
good reduction at any $p\neq l$, and are asymptotically optimal over $\F_{p^2}$. 
The curves $X_0(l^i)_\Q$ form a tower over $\Q$ that is recursively determined from the 
two first steps (actually the first step is enough, see historical notes and references below). 
More precisely, the tower is deduced by iterated fiber products from the two following data:
\begin{itemize}
\item the canonical morphisms over $\Q$
$$X_0(l^2)\rightarrow X_0(l) \rightarrow X_0(1) $$
\item the Atkin-Lehner involutions $w_i$ on $X_0(l^i)_\Q$ for $i=0,1,2$
\end{itemize}

\begin{rema}\label{RemMat}
Actually the first step are enough to deduce the whole tower recursively 
(see historical notes and references below).
Namely, one needs only the covering map $X_0(l)\rightarrow X_0(1)$ and 
the Atkin-Lehner involutions $w_i$, for $i=0,1$. Caution must be taken since the fiber product of the first 
step $X_0(l)$ with its Atkin-Lehner twist ---in addition to be highly singular--- contains a 
second irreducible component in addition to $X_0(l^2)$. This comes from degree reasons, \cite[VI §2.3 \& §3.2]{ramb1} 
(or modular interpretation reasons, if one prefers).
\end{rema}

The genera in a single tower $X_0(l^i)_{\F_p}$ for any $p$ are tightly controled by the prime powers $l^i$:
\begin{equation}\label{eq:generaXN} l^i(1+1/l)/12+o(g_i)\leq g_i \leq l^i(1+1/l)/12 \end{equation}
(see \cite[4.1]{tsvl2} or \cite[Th 3.1.1 \& p107]{ds}). So this single tower does not form a dense family.

Now let $l'\neq l$ be another prime and consider the recursive tower $X_0(l'^j)_\Q$. Both towers are defined 
over the same basis $X_0(1)$, and, by taking fiber products over $X_0(1)$, we obtain:
$$X_0(l^i)_\Q\times X_0(l'^j)_\Q=X_0(l^il'^j)_\Q$$
for any $i$ and $j$. By doing so for every indexes $i$ and $j$ we obtain the family $\bigr\{X_0(l^il'^j)_\Q\bigl\}_{i,j}$: let us call this family the "intertwinning" of the two recursive towers. This family has good reduction at any prime $p\neq l,l'$ and is asymptotically optimal. 
The genera in this family are now closely controled by the prime products $l^il'^j$, as follows from
\begin{equation}\label{eq:generaXij} l^il'^j(1+1/l)(1+1/l')/12+o(g_{i,j})\leq g_{i,j} \leq l^il'^j(1+1/l)(1+1/l')/12 \; .\end{equation}
The key observation is that the family of integers $l^il'^j$ is \emph{dense}, i.e. its growth rate tends to zero. So that the intertwinned family $\bigr\{X_0(l^il'^j)_\Q\bigl\}_{i,j}$ is dense. 

\subsubsection{Problems of descent on Shimura curves and open questions}

Let us shift to Shimura curves and consider three specific recursive towers $X_0(\p^i)$ defined over the same basis $X_0(1)$ of genus zero. Let $F = \Q[cos(2\pi/7)]$ be the totally real number field of degree three, $\p_2$ and $\p_3$ the prime ideals over the inert primes  $(2)$ and $(3)$ and $\p_7$ the prime ideal over the split prime $(7)$. Let $B$ be the quaternion algebra over $F$, which is ramified exactly at two of the three real places and no finite place. $B$ contains one unique conjugacy class of Eichler orders of given level. In particular, "the" maximal order $\cO$ has its group of units $\cO^1$ which embeds into $\mathrm{PSL}_2(\R)$ onto the celebrated $(2,3,7)$ triangle group (it is the hyperbolic group of smallest covolume). The Shimura curve $X_0(1)_\C$ uniformized by this group has a canonical model over $F$ of genus zero with three rational points, which precisely arise from the elliptic points, of orders $2$, $3$ and $7$. Above this base curve one has notably the three towers $X_0(\p^i)$ where $\p=\p_2,\p_3$ and $\p_7$, which have canonical models over $F$. They have good reduction at every prime $\p'$ of $F$ different from  $\p_2$, $\p_3$ and $\p_7$ and, if furthermore $\p'=(p)$ comes from an inert prime, then the reductions $X_0(\p^i)_{\F_{p^3}}$ modulo $\p$ have an \emph{asymptotically optimal} number of points over $\F_{p^6}$ (see \cite[Th IV.4.5]{duce}, which is established from two independent methods). 

Now, intertwinning the two towers $X_0(\p_2^i)$ and $X_0(\p_7^j)$ over $X_0(1)$ gives a dense family $\bigl\{X_0(\p_2^i\p_7^j)_F\bigr\}_{i,j}$ over $F$, with genera tightly controled by the products $8^i.7^j$:
\begin{equation} \label{eq:generaX27}  g_{i,j}= 7^{j-2}(8^{i-1}  6/7 + 1/7) \; \text{for $i\geq 1$ and $j\geq 2$} \end{equation}
(and similar formulas for smaller $i$ or $j$: see \cite[IV Corollary 2.12]{ramb1}). In particular it has good reduction modulo $p_3=(3)$ and yields an asymptotically optimal dense family $X_0(\p_2^i\p_7^j)_{\F_{3^3}}$ over $\F_{3^3}$ with many points in $\F_{3^6}$. Now, the interesting problem for bilinear multiplication over $\F_3$ is: \emph{can we descend this family over $\F_3$ ?} Much of the work towards this result has been done, since it is proven in \cite[VI §5.2]{ramb1} that the two first steps of the reductions modulo $\p'=p_3$ of the two towers descend over $\F_3$. But recall that, over $F$, these two first steps are sufficient to build the whole family. So, the problem of descent of the family over $\F_3$ falls back to the following general question:

\begin{op} 
\begin{newconj} 
\label{conj:recmodp} Are good reductions of towers of Shimura curves recursive ? 
\end{newconj}

We are confident that this point falls back to the modular interpretation of \emph{integral models} of Shimura curves ---and not only models over number fields, such as $F$---, which should be also well known to specialists.

Additional evidence supports the descent question that we are concerned with, since it is also established in \cite[Th. V.5.14]{ramb1} that the family $\bigl\{X_0(\p_2^i\p_7^j)_F\bigr\}_{i,j}$ descends over $\Q$, and that strong numerical evidence (the number of points) suggests that the third steps also descend (\cite[VI §5.2]{ramb1}).

Recapitulating: descent of the previous family, as would be implied e.g. by Conjecture \ref{conj:recmodp}, would provide a dense family over $\F_3$ with many points of degree $6$, which would thus establish:
\begin{equation} \label{eq:thB} \betad_6(3)=\frac{3^3-1}{6} \end{equation}
which is (prematurely) claimed as "Theorem B" in \cite{ramb1}.
\end{op}

Likewise, intertwinning the two towers $X_0(\p_3^i)$ and $X_0(\p_7^j)$ over $X_0(1)$ gives a dense family $\bigl\{X_0(\p_3^i\p_7^j)_F\bigr\}_{i,j}$ over $F$, with genera tightly controled by the products $27^i.7^j$, good reduction modulo $p_2=(2)$ over $\F_{2^3}$ and asymptotically many points in $\F_{2^6}$.

\begin{op}
Similarly, we are concerned with descent of this dense family over $\F_2$, which if true would thus yield the value $\betad_6(2)=\frac{2^3-1}{6}$. 
Let us assume that the previous Conjecture \ref{conj:recmodp} is true: then this would already imply that the tower $X_0(\p_7^j)$ descends over $\F_2$. So, we would then be left to show that the two first steps of the tower $X_0(\p_3^i)$ also descend. More precisely: 

\begin{newconj}\label{conjX} The following morphisms descend over $\F_2$:
the canonical branched cover $X_0(\p_3^2)_{\F_{2^3}} \longrightarrow X_0(\mathfrak{p}_3)_{\F_{2^3}}$, 
and the Atkin Lehner involution on $X_0(\mathfrak{p}^{2}_3)_{\F_{2^3}}$ .
\end{newconj}
Finally, notice that the first step of this tower $X_0(\mathfrak{p}_3)_\Q \longrightarrow X_0(1)_\Q$ was explicitly 
computed over $\Q$ in \cite{elki2006}: a Belyi map of degree $27$. 
So, if it was true that good reduction of towers of Shimura curves were also recursive from the first step (see Remark \ref{RemMat}), then one would be left with the easier problem of finding a good reduction modulo $(3)$ of this Belyi map of degree $27$.
\end{op}

\begin{op} 

From a more general point of view, the so far known families of curves attaining the Drinfeld Vladuts bound over q are all defined over fields of square cardinal $q=p^(2t)$. The following conjecture states (under an equivalent form) that for all square $q$, there exists such a dense optimal family over $\F_q$ which descend over the prime field $\F_p$.

 \begin{newconj}\label{conjY}
 Let $p$ be a prime number and $2t\geq 4$ an even integer. Then the following equality holds: 
 \begin{equation}\label{conjYformula}
 \betad_r(q)=\frac{p^t-1}{2t}.
 \end{equation}
 Said otherwise: there exists a family $(X_s/\F_{p^{2t}})_{s\geq1}$ of curves over $\F_p$ with 
 (increasing) genera $g_s$ tending to infinity such that
\begin{enumerate}
\item[(i)] $X_s$ is, actually, defined over the prime field $\F_p$;
\item[(ii)] $\lim_{s \rightarrow \infty}\frac{g_{s+1}}{g_s}=1$ (maximal density condition)
\item[(iii)] $\lim_{s \rightarrow \infty}\frac{\mid X_s(\F_{p^{2t}})\mid}{g_s}=p^t-1$ (Ihara constant over $\F_{p^{2t}}$)
\end{enumerate}
 \end{newconj}
 
\end{op}

\begin{op}

The following conjecture was proposed in \cite{randITW}, to which we added a density requirement.

\begin{newconj}\label{conjZ} Let $p>2$ be an odd prime. Then there exists a sequence of numbers 
$(N_s)_s$, with $\lim_{s\rightarrow \infty}\frac{N_{s+1}}{N_s}=1$  
(density condition), such that Hecke operator $T_p(N_s)$ acting on the space of weight $2$ cusp forms $S_2(\Gamma_0(N_s))$,
has an odd determinant.
\end{newconj}

Its consequence would be the asymptotic vanishing of two-torsion in classical modular curves:

\begin{newpropo} \label{prop:notwotors} Under Conjecture \ref{conjZ}, then there exists a dense family 
of (classical modular) curves $\{X_0(N_s)/\F_p\}_s$ such that $$\bigl(\mathrm{Cl}_0(X_0(N_s)\bigr)(\F_{p^2})[2]=\{0\} $$
(i.e. that have no two torsion in their class group.)\end{newpropo} 

This proposition is stated as Conjecture I 2.8 in \cite{ramb1}. Here, a detailed proof that it results 
from Conjecture \ref{conjZ} is given: in the discussion above Conjecture I 2.8 and, also, in §II.5 
(for the key formula (2.6)). The following practical consequence will be proven in the Annex.

\begin{newpropo} \label{prop:thAconjZ} Let $p$ be a prime number such that Conjecture 
\ref{conjZ} holds for $p$, and $r$ an integer such that 
$\bigl\{q=p \hbox{ and } r=2 \bigr\}$ \hbox{ or } $\bigl\{ q=p^2 \hbox{ and } r=1 \bigr\}$, 
then formula (a) in Theorem \ref{TheoA-Sym} also holds.
\end{newpropo}
\end{op}

\subsubsection{References and historical notes for section 6}

\phantom{a}\hfill

\emph{Recursive modular towers:} The recursivity of towers of classical modular curves was pointed 
in the seminal paper of N. Elkies \cite[pp 1-3]{elkiRec}, where more details and a proof over $\C$ can be found. 
The proof carries over the canonical models 
over $\Q$ since the moduli interpretation in terms of elliptic curves is the same. 
N. Elkies also claims --and uses-- that towers of Shimura curves are recursive. The proof of 
this fact is formally analogous: see \cite[Proposition IV.5.1]{duce}. But actually, 
extra care must be taken with the irreducibility of the 
tensor products involved: \cite[VI §2.3 \& §3.2]{ramb1}, 
because the moduli interpretation is much more complicated. \emph{Intertwinning two towers} over the same basis: 
this construction is already mentionned in \cite[top of page 7]{elkiRec}. 
The crucial observation that the resulting family is dense was pointed to us by N. Elkies in August 2015. 

\emph{Recursivity from the first step:} The fact that the first step of modular towers is actually 
enough to construct them recursively is already pointed in \cite[footnote 4]{elkiRec} 
and \cite[p8]{elki}, and brought to our attention by N. Elkies in 2017.

\emph{About conjecture \ref{conjY}:} this conjecture was essentially stated as a Lemma IV.4 in [39] . For their proof, the authors claim that some specific Shimura curves, with Galois invariant parameters, descend over the rationals. This claim is unfortunately false: in [23, §3] we exhibited counterexamples to this claim, which evidence more generally that Shimura curves do not descend over their field of moduli. Consequences of Conjecture  \ref{conjY} on upper-limit asymptotic complexities are given M. Rambaud in \cite[Table 2.2]{ramb1}, lines "Conj Y". Notice that they improve a bit those claimed by \cite{cacrxiya}, displayed in footnote \ref{URcacrxiya} page \pageref{URcacrxiya}.


%

\emph{More on explicit computations:} Since the seminal works of \cite{tvz} and \cite{ihar} 
on Shimura curves with many points, many equations of curves of genus zero and one were computed 
in \cite{elkiScc}, \cite{hall} and \cite{sij}. Further examples of recursive towers of Shimura curves 
can be found in: \cite[IV Example 5.3]{duce}; \cite{haserec}; \cite[VI §3]{ramb1} (defined over 
a totally real field of narrow class number two, with a record number of points over $F_{5^4}$ in genus 5). 
The (nonexplicit) list of Shimura curves of genus less than two can be found in \cite{voightGenus2}. From this 
data and the recent tools for Belyi maps developped in \cite{voightBelyi}, one could access the dozen of recursive 
towers whose first step are covering map of $\P^1$ of degree $\leq 9$ ramified above three points. Finally, 
when the first step is over a \emph{genus one} curve, then a first example was computed in C. Levrat's masters thesis \cite{levr}.

\section{Obtaining a divisor of optimal degree for symmetric algorithms}\label{sym_methods}

Using the numerical criteria at the end of Theorem~\ref{theo_evalder},
in the symmetric case $\D_1=\D_2$, we meet the following problem:
given
\begin{itemize}
\item $q$  a prime power
\item $F/\F_q$ a function field, of genus $g$
\item $\mathcal{Q}$ a divisor of $F/\F_q$, of degree $n=\deg\mathcal{Q}$
\item $\mathcal{G}$ a divisor of $F/\F_q$, of degree $N=\deg\mathcal{G}$
\end{itemize}
does there exist a divisor $\D$ such that the two conditions
\begin{equation}
\label{cond_nonspecial}
i(\D-\mathcal{Q})=0
\end{equation}
and
\begin{equation}
\label{cond_sanssection}
\dim\mathcal{L}(2\D-\mathcal{G})=0
\end{equation}
are both satisfied?

Clearly the answer will depend on $n$ and $N$. By Riemann-Roch's theorem,
condition \eqref{cond_nonspecial} implies $\deg\D-n\geq g-1$ and
condition \eqref{cond_sanssection} implies $2\deg\D-N\leq g-1$,
so combining both we see
\begin{equation}
\label{optimal_degree}
N\geq 2n+g-1
\end{equation}
is a necessary condition for the existence of a solution.

Observe that, in order to get the algorithm of best complexity for given $n$, we need $N$ to be as small as possible.

In their original paper \cite{chch}, D.V. Chudnovsky and G.V. Chudnovsky introduced
a simple cardinality and degree argument,
later made more explicit by S. Ballet in \cite{ball1},
which proved the existence of a solution under the less optimal condition
\begin{equation}
\label{suboptimal_degree}
N\geq 2n+2g-1.
\end{equation}

As explained in section~\ref{mM}, Shparlinski-Tsfasman-Vladut tried to improve
the original bound of Chudnovsky-Chudnovsky by proving the existence of $\D$
under the optimal condition \eqref{optimal_degree}, instead of \eqref{suboptimal_degree}.
For this they had to adapt the cardinality argument, but they failed to notice the
consequence of the existence of $2$-torsion in the class group when dealing with \eqref{cond_sanssection}.

In order to repair their proof, two approaches were devised:
\begin{itemize}
\item choose curves with $2$-torsion as small as possible
\item directly construct $\D$ under condition \eqref{optimal_degree}.
\end{itemize}

\subsection{Bounding the $2$-torsion} \label{Bound2torsion}

Bounds on torsion in the class group were first introduced in a very similar context,
that of frameproof codes (also called linear intersecting codes), by C. Xing \cite{Xing}.
Indeed, in order to obtain a $s$-frameproof code of high rate, one needs, given a
divisor $\mathcal{G}$, to prove the existence of a divisor $\D$ of high degree
such that
\begin{equation}
\label{cond_frameproof}
\dim\mathcal{L}(s\D-\mathcal{G})=0.
\end{equation}
C. Xing proved the existence of such a $\D$ using a cardinality argument similar to that
of Chudnovsky-Chudnovsky and Shparlinski-Tsfasman-Vladut, while correctly recognizing
the difficulty with $s$-torsion.
His result on the rate of $s$-frameproof codes thus includes a term accounting for the
size of the $s$-torsion subgroup.
Actually, C. Xing used the well known upper bound $s^{2g}$
for the size of the $s$-torsion subgroup in the Jacobian of curve of genus $g$.

It is natural to ask for better bounds, especially in the asymptotic case $g\to\infty$.
This problem was formalized and studied, independently,
\begin{itemize}
\item by H. Randriambolona, through the quantity $\delta_s^-(q)$ in \cite{randITW}
\item by I. Cascudo, R. Cramer and C. Xing, through the torsion-limit $J_r(q,a)$ in \cite{cacrxi}\cite{cacrxi2}.
\end{itemize}

One of the questions asked by H. Randriambolona in \cite{randITW} is the following: for given $q$ and $s$,
can one find an infinite sequence of curves having many rational points (ideally, matching the Ihara constant $A(q)$),
but whose class group has few $s$-torsion?

How asymptotically small this $s$-torsion can be is measured by the following quantity:
\begin{defi}\label{Ihara}
Let $\delta_s^-(q)$ be the smallest real number such that there exists a sequence $(\mathcal{X}_k)_{k\geq1}$
of curves over $\F_q$, of increasing genus $g_k=g(\mathcal{X}_k)$, having an asymptotically number of rational points:
$$\lim_{k\rightarrow\infty} \frac{|\mathcal{X}_k(\F_q)|}{g_k}= A(q)$$
and such that the cardinal of the $s$-torsion subgroup $\mathcal{J}_k(\F_q)[s]$ of the group of rational points over $\F_q$ of the Jacobian 
$\mathcal{J}_k=\mathcal{J}(\mathcal{X}_k)$
satisfies
$$\lim_{k\rightarrow\infty}\frac{\log_s|\mathcal{J}_k(\F_q)[s]|}{g_k}= \delta_s^-(q).$$ 
\end{defi}

\begin{op}\label{optorsionlimit}
Estimation of the quantity $\delta_s^-(q)$ for an infinite sequence of curves attaining the Drinfeld-Vladut bound.
H. Randriambololona conjectures that $\delta_s^-(q)=0$ for all $s$ and $q$,
i.e. that there exists curves that have an asymptotically maximal number of points over $\F_{q}$
and whose class groups have asymptotically negligible $s$-torsion.
Of special importance for us is the case $s=2$, i.e. the case of $2$-torsion.
In \cite{randITW} H. Randriambololona puts focus on classical modular curves, which have an asymptotically maximal number of points over $\F_{p^2}$ (for $p$ prime). The size of the class group of such a curve is given by the determinant of a Hecke operator.
This leads to deep number theoretic questions on the parity of these determinants,
which remain conjectural at this time.
\end{op}

\vspace{\baselineskip}

In \cite{cacrxi2}, I. Cascudo, R. Cramer and C. Xing generalize conditions like
\eqref{cond_nonspecial}\eqref{cond_sanssection} or like \eqref{cond_frameproof}
into what they name Riemann-Roch systems of equations.
They adapt the cardinality argument of \cite{chch}\cite{shtsvl}\cite{Xing} in this more general framework.
First, for a function field $F/\F_q$, let $\mathcal{J}_F$ be its zero divisor class group.
Let then $\mathcal{J}_F[r]$ be its $r$-torsion subgroup, of cardinality $J_F[r]=\mid \mathcal{J}_F[r]\mid$.
Their main result (see \cite[Theorem 3.2]{cacrxi2}) is as follows :
\begin{propo}
Let:
\begin{itemize}
\item $q$ be a prime power
\item $F/\F_q$ be a function field
\item $h$ be the class number of $F$  
\item $A_m$ the number of effective divisors
of degree $m$ in the group of divisors $Div(F)$ for $m>0$
\item $u\geq1$ be an integer
\item $\mathcal{Y}_1,\dots,\mathcal{Y}_u$ be divisors of $F$
\item $m_1,\dots,m_u$ be nonzero integers.
\end{itemize}
Suppose that for some integer $s\in\Z$, the inequality
$$h>\sum_{i=1}^uA_{r_i(s)}J_F[m_i]$$
holds, where $r_i(s)=m_is+\deg\mathcal{Y}_i$.
Then the system of conditions
$$\dim\mathcal{L}(m_1\D+\mathcal{Y}_1)=\dots=\dim\mathcal{L}(m_u\D+\mathcal{Y}_u)=0$$
is satisfied by some divisor $\D$ of degree $s$.
\end{propo}

In order to measure the size of the torsion subgroups, they introduce the notion of torsion-limit:
\begin{defi}\label{def_torsion-limit}
For each family $\mathcal{F}=\{F/\F_q\}$ of function fields with increasing genus $g(F)$, 
we define the asymptotic limit
$$J_r(\mathcal{F})=\liminf_{F \in \mathcal{F}} \frac{\log_q J_F[r]}{g(F)}.$$
For a prime power $q$, an integer $r>1$ and a real number $a \leq A(q)$, let $\Upsilon$
be a set of families $\{\mathcal{F}\}$ of function fields over $\F_q$ such that
the genus in each family tends to $\infty$ and the Ihara limit satisfies ${A(\mathcal{F}) \geq a}$
for every $\mathcal{F}\in \Upsilon$.
Then the asymptotic quantity $J_r(q,a)$ is defined by
$$J_r(q,a)=\liminf_{\mathcal{F} \in \Upsilon}J_r(\mathcal{F}).$$
\end{defi}

Thanks to the equivalence between curves and function fields, where the group of rational
points of the Jacobian corresponds to the zero divisor class group, we see that this
torsion-limit is related to the constant $\delta_r^-(q)$ by the relation:
\begin{equation}
J_r(q,A(q))=\log_q(r)\delta_r^-(q).
\end{equation}

This torsion-limit can be introduced as a correcting term in the denominator of the bound claimed by Shparlinski, Tsfasman, and Vladut, as we will see in Section \ref{msMs}.

However, another approach is possible namely the direct construction.

\subsection{Direct construction}\label{construction_sym}

The direct construction consists on finding the best divisors $D$ to apply CCMA, i.e divisors $D$ satisfying Conditions (\ref{cond_nonspecial}) and (\ref{cond_sanssection}) for given $q$ and $n$.
The idea is explicitly introduced by S. Ballet in \cite[Theorem 2.2]{ball4} as we will see more precisely in Section \ref{mus}.  
Then J. Chaumine proved in \cite{chau0} (cf. also \cite{chau1}) that the direct construction is optimal in the elliptic case, improving then the result of A. Shokrollahi \cite{shok} as we will see in Section \ref{musexact}.
Then, H. Randriambolona introduces news ideas which originate in his work \cite{randIJM} for the construction of intersecting codes.
The technique was then extended in \cite{rand2D-G} in order to solve more general
Riemann-Roch systems of equations.
In the case of the Riemann-Roch system associated with a CCMA, it allows the
effective construction of a solution, in most cases up to optimal degree.

The key point is the following result \cite[Lemma 9]{randIJM}, which can be seen as a numerical variant of a generalized Pl\"ucker formula:
\begin{lemm}
\label{l(A+2P)=l(A)}
Let $X$ be a curve of genus $g$ over a perfect field $K$,
and let $\mathcal{A}$ be a divisor on $X$ with $\deg\mathcal{A}\leq g-3$
and
$$\dim\mathcal{L}(\mathcal{A})=0.$$
Then for all points $P\in X(K)$ except perhaps for at most $4g$ of them,
we have
$$\dim\mathcal{L}(\mathcal{A}+2P)=0.$$
\end{lemm}

In \cite{rand2D-G} it is shown how the bound $4g$ can be slightly improved when $K$ is a finite field.
However the original Lemma~\ref{l(A+2P)=l(A)} suffices to prove the following result \cite[Corollary 20]{rand2D-G}: 
\begin{propo}\label{propo_construction_sym}
Let:
\begin{itemize}
\item $q$ be a prime power
\item $F/\F_q$ be a function field, of genus $g$
\item $\mathcal{Q}$ be a divisor of $F/\F_q$, of degree $n=\deg\mathcal{Q}$
\item $\mathcal{G}$ be a divisor of $F/\F_q$, of degree $N=\deg\mathcal{G}$.
\end{itemize}
Assume that the number of degree $1$ places of $F$ satisfies
$$N_1(F/\F_q)>5g.$$
Then, provided
$$N\geq 2n+g-1$$
there exists a divisor $\D$ of $F/\F_q$ such that $\D-\mathcal{Q}$ is nonspecial
of degree $g-1$ and $2\D-\mathcal{G}$ is zero-dimensional:
\begin{itemize}
\item $\deg\D=n+g-1$
\item $\dim\mathcal{L}(\D-\mathcal{Q})=0$
\item $\dim\mathcal{L}(2\D-\mathcal{G})=0$.
\end{itemize}
\end{propo}
Observe that for a divisor of degree $g-1$, nonspecial and zero-dimensional are equivalent,
so here $i(\D-\mathcal{Q})=0$ and $\dim\mathcal{L}(\D-\mathcal{Q})=0$ are equivalent.

Observe also that Proposition~\ref{propo_construction_sym} gives precisely what was required in
the approach of Shparlinski, Tsfasman and Vladut, as described in section \ref{mM}, 
with $N=2n+g-1$, $\mathcal{Q}=Q$, and $\mathcal{G}=P_1+\cdots+P_N$.
The only downside is the condition that $F$ should have sufficiently many rational places.

Beside \cite{rand2D-G},
the proof of this Proposition~\ref{propo_construction_sym} can also be found inside the proof of \cite[Theorem 5.2(c)]{randJComp}.

\section{Asymptotic upper bounds}  \label{AsyBounds} 

The asymptotic study of the bilinear complexity of the multiplication consists on evaluating the quantities $m_q$, $M_q$, $\ms_q$, $\Ms_q$. 
The importance of this study comes from the fact that generally we have better estimations of these quantities than those of the constants 
$C_q$ and $C^{sym}_q$. Indeed, the best known families of curves suitable to the application of the D. V. and G. V. Chudnovsky algorithm 
are known asymptotically, in particular the families of Shimura curves used by I. Shparlinski, M. Tsfasman and S. Vladut in \cite{shtsvl}.
These latter establish the following general result which we can see as a direct consequence 
of Lemma \ref{plongement} (or of \cite[Lemma 1.2]{shtsvl})\footnote{Their main motivation to introduce this lemma was, 
from the finiteness of $M_q$ for $q$ square, to deduce finiteness of $M_q$ for all $q$.}.  

\begin{lemm}\label{lemasyMqmq}
For any prime power $q$ and any positive integer $n$ we have 
\begin{equation}\label{plongementasymptolimiteinf}
m_{q}\leq m_{q^n}\cdot\mu_{q}(n)/n
\end{equation}
\begin{equation}\label{plongementasymptolimitesup}
 M_{q}\leq M_{q^n}\cdot\mu_{q}(n).
\end{equation}
\end{lemm}

Actually, inequality \eqref{plongementasymptolimiteinf} about $m_q$ is already implicit
in the original paper of D. V. Chudnovsky and G. V. Chudnovsky (from \cite[eq. (6.2)]{chch}).
So, here, the important new contribution of I. Shparlinski, M. Tsfasman and S. Vladut is inequality \eqref{plongementasymptolimitesup} about $M_q$.
Note that these inequalities are also true in the symmetric case, as a consequence of Lemma \ref{plongementsym}:

\begin{lemm}\label{lemsymMqmq}
\begin{equation}\label{plongementasymptolimiteinfsym}
\ms_{q}\leq \ms_{q^n}\cdot \mus_{q}(n)/n
\end{equation}
\begin{equation}\label{plongementasymptolimitesupsym}
 \Ms_{q}\leq \Ms_{q^n}\cdot\mus_{q}(n).
\end{equation}
\end{lemm}

By using Theorem \ref{thm_wdg} with Lemma \ref{lemasyMqmq} or Lemma \ref{lemsymMqmq}, 
we trivially get the following useful corollary: 

\begin{coro}\label{corotrivialmq}
For every prime power $q$, we have $m_q\leq\frac{3}{2}m_{q^2}$, $\ms_q\leq\frac{3}{2}\ms_{q^2}$, $M_q\leq 3M_{q^2}$, and $\Ms_q\leq 3\Ms_{q^2}$.
If ${q \geq 4}$, then $m_q\leq\frac{5}{3}m_{q^3}$, $\ms_q\leq\frac{5}{3}\ms_{q^3}$, $M_q\leq 5M_{q^3}$, and $\Ms_q\leq 5\Ms_{q^3}$.
\end{coro}

Let us recall that $A(q)$ denotes the Ihara limit defined by 
$A(q):=\limsup_{g\rightarrow \infty} \frac{N_q(g)}{g}$ where $N_q(g)$ is the maximum number 
of rational places over all the algebraic function fields over $\F_q$ of genus $g$ (cf. also Definition \ref{Ihara}).

\subsection{Upper bounds on $m_q$ and $M_q$}\label{mM1}

Thanks to the asymmetric interpolation allowed by the generalized CCMA (cf. Section \ref{theo_evalder}), 
H. Randriambololona \cite[Theorem 6.3 and Theorem 6.4]{randJComp} obtains bounds for $m_q$ and $M_q$.
 For $m_q$, the bound reads:
\begin{theo}
\label{STV_m_repare}
Let $q$ be a prime power such that $A(q)>1$. Then
\begin{equation}
m_q\leq2\left(1+\frac{1}{A(q)-1}\right).
\end{equation}
\end{theo}
For $M_q$, it reads:
\begin{theo}
\label{STV_M_repare}
Let $q=p^{2r}\geq 9$ be a square prime power. Then
\begin{equation}
M_q\leq2\left(1+\frac{1}{\sqrt{q}-2}\right).
\end{equation}
\end{theo}

Combined with Lemma \ref{lemasyMqmq} and $\mu_q(2)=3$, this implies at once:
\begin{coro}
\label{coro_STV_repare}
Let $q\geq3$ be a prime or a nonsquare prime power. Then
\begin{equation}
m_q\leq3\left(1+\frac{1}{q-2}\right)
\end{equation}
and
\begin{equation}
M_q\leq6\left(1+\frac{1}{q-2}\right).
\end{equation}
\end{coro}

Moreover, from Theorem \ref{bornesasympt}, J. Pieltant and H. Randriambololona deduce the following asymptotic bounds in the general case:

\begin{theo}
$$
\begin{array}{lllll}
M_3 \leq 6 & M_4 \leq \frac{87}{19}\simeq  4.579 &M_5 \leq 4.5 & M_{11} \leq 3.6 & M_{13} \leq 3.5.
\end{array}
$$
\end{theo}

These bounds are the best published current asymptotic bounds in the general case. They are deduced from the best known uniform bounds.
Indeed, the purely asymptotic bounds\footnote{\samepage These unproven bounds are:
$$
M_q\leq\frac{2\mu_q(t)}{t}\left(1+\frac{1}{q^{t/2}-2}\right)
$$
for $q$ be a prime power and $t\geq1$ an integer such
that $q^t\geq9$ is a square;\\
and
$$
M_2\leq \frac{35}{6}, \qquad M_3\leq \frac{36}{7}, \qquad M_4\leq \frac{30}{7}, \qquad M_5\leq 4, \qquad
M_7\leq 3.6, \qquad
M_8\leq 3.5.
$$
}
 given in Theorem 5.3, Corollary 5.4, Corollary 5.5 of \cite{pira} are unproved as established in \cite{bapirasi}.
 In addition, as corollary of uniform bounds in Theorem \ref{newboundphr} (cf. Section \ref{mu}), 
 H. Randriambololona obtains recently the following result:
 
 \begin{theo}
For $p\geq 7$, we have: $$M_p\leq 3\left(1+\frac{1}{p-2}\right).$$
\end{theo}

Finally,  in \cite{ramb1} M. Rambaud obtains the current best general upper-limit asymptotic bound, namely:

\begin{theo}\label{TheoA-Asym}
Let $q$ a prime power and $r\geq 1$, $l\geq 1$ be two positive integers. Then, as long as $rl\betad_r(q)-1>0$, we have:
$$M_q\leq \frac{2\mu_q(r,l)}{rl} \left(1+\frac{1}{rl\betad_r(q)-1}\right). $$
\end{theo}

In particular, this result enables to obtain the following value (with $(r,l)=(4,1)$, $\mu_q(r,l)\leq \mus_q(r,l)= 9$ by Table \ref{ExactBilinComp} and $\betad_r(2)=\frac{3}{4}$ by Formula (\ref{Atilder}):

\begin{coro}\label{M_2Ramb}
$$M_2\leq 7.$$ 
\end{coro}

\subsection{Upper bounds on $\ms_q$ and $\Ms_q$} \label{msMs}


Initially, by using the original Chudnovsky and Chudnovsky, I. Shparlinski, M. Tsfasman and S. Vladut \cite{shtsvl} 
obtain upper bounds\footnote{
These are following bounds: 
$$\ms_q\leq 2\left(1+\frac{1}{A(q)-1}\right),$$ where $A(q)>1$ is defined in Proposition \ref{mqwithAq},  
$$\ms_q\leq 2\left(1+\frac{1}{\sqrt{q}-2}\right),$$ where $q$ is a perfect square $\geq 9,$ 
$$\ms_q\leq 2\left(1+\frac{1}{c\log_2{q}-1}\right),$$ where $q\geq 2^{1/c}$ with $c$ is a positive constant, 
$$\ms_q\leq 2\left(1+\frac{q^{1/3}+2}{2q^{2/3}-q^{1/3}-4}\right),$$ $$\ms_2\leq \frac{35}{6},$$ $$\ms_q\leq 3\left(1+\frac{1}{q-2}\right),$$ 
where $q>2,$ $$\Ms_q\leq 2\left(1+\frac{1}{\sqrt{q}-2}\right),$$ 
where $q\geq 9$ is a perfect square, $$\Ms_q\leq 6\left(1+\frac{1}{q-2}\right),
$$ where $q>2$, and $$\Ms_2\leq 27$$
given respectively in \cite[Theorem 3.1]{shtsvl}, \cite[Corollary 3.4]{shtsvl},  \cite[Corollary 3.5]{shtsvl}, \cite[Remark 3.6]{shtsvl}, 
\cite[Corollary 3.7]{shtsvl}, 
\cite[Corollary 3.8]{shtsvl}, \cite[Theorem 3.9]{shtsvl} and \cite[Corollary 3.10]{shtsvl} for the last two bounds.
 Note that these bounds are originally formulated with notation $m_q$ and $M_q$, but for the same reasons that those mentioned in footnote \ref{symornot} of Section \ref{mM}, 
 these bounds concern the quantities $\Ms_q$ and $\ms_q$. Note that there exist proved bounds exceeding the last bound (cf. Proposition \ref{newbound4}).} 
of $\Ms_q$ and $\ms_q$ for any $q$, which are not completely proved because of the gap mentioned in Section \ref{mM}.
H. Randriambololona in \cite[Theorem 6.3 and Theorem 6.4]{randJComp} obtains the following results which prove the bounds 
of Shparlinsky-Tsfasman-Vladut with a slight restriction on the range of the values for $A(q)$ and $q$.
For $\ms_q$, the bound reads:
\begin{theo}
\label{STV_ms_repare}
Let $q$ be a prime power such that $A(q)>5$. Then
\begin{equation}
\ms_q\leq2\left(1+\frac{1}{A(q)-1}\right).
\end{equation}
\end{theo}
For $\Ms_q$, it reads:
\begin{theo}
\label{STV_Ms_repare}
Let $q=p^{2r}\geq 49$ be a square prime power. Then
\begin{equation}
\Ms_q\leq2\left(1+\frac{1}{\sqrt{q}-2}\right).
\end{equation}
\end{theo}

Combined with Lemma \ref{lemsymMqmq} and $\mus_q(2)=3$, this implies at once:
\begin{coro}
\label{coro_STV_sym_repare}
Leq $q\geq7$ be a prime or a nonsquare prime power. Then
\begin{equation}
\ms_q\leq3\left(1+\frac{1}{q-2}\right)
\end{equation}
and
\begin{equation}
\Ms_q\leq6\left(1+\frac{1}{q-2}\right).
\end{equation}
\end{coro}

\vspace{1em}

%
%
%

In \cite{bachpi}, S. Ballet, J. Chaumine and J. Pieltant obtain bounds slightly less accurate than the bounds of the above results but for a slightly larger range of values for $A(q)$ and $q$.
They give the following propositions.

\begin{propo}\label{mqwithAq}
 Let $q$ be a prime power such that $A(q)>2$. Then
$$\ms_q \leq 2\left(1+\frac{1}{A(q)-2}\right).$$
\end{propo}

\begin{coro}\label{coromq1}
Let $q=p^{2m}$ be a square prime power such that $q \geq 16$.
Then
$$\ms_{q}\leq 2\left(1+\frac{1}{\sqrt{q}-3}\right).$$
\end{coro}

Note that this corollary slightly improves the range of the bound (\ref{chudmq}) proved by D.V. and G.V. Chudnovsky. 
Now in the case of arbitrary $q$, they obtain:

\begin{coro}\label{coromq2}
For any $q=p^m>3$,
$$\ms_{q}\leq 3\left(1+\frac{1}{q-3}\right).$$
\end{coro}

Moreover, for $M^{sym}_q$ they obtain the same value for the same range than that of  $m^{sym}_q$:

\begin{propo}\label{newbound}
Let $q=p^{2m}$ be a square prime power such that $q \geq 16$. Then
\begin{equation}\label{formulenewbound}
\Ms_{q}\leq 2\left(1+\frac{1}{\sqrt{q}-3}\right).
\end{equation}
\end{propo}

\begin{propo}\label{newbound2}
Let $q=p^m$ be a prime power with odd $m$ such that $q \geq 5$ .
Then
\begin{equation}\label{formulenewbound2}
\Ms_{q}\leq 3\left(1+\frac{2}{q-3}\right).
\end{equation}
\end{propo}

\begin{rema}
For $q$ square, Bound (\ref{formulenewbound}) is better that Bound (\ref{formulenewbound2}) except for $q=16$. 
\end{rema}

When $q$ is a prime number, the uniform bounds of Proposition \ref{newboundpbz} obtained in \cite[Proposition 10]{bazy2} by S. Ballet and A. Zykin lead to the asymptotic symmetric complexity given in the following proposition:

\begin{propo}\label{newbound3}
Let $p \geq 5$ be a prime number.
Then
\begin{equation}\label{formulenewbound3}
\Ms_{p} \leq 3\left(1+\frac{\frac{4}{3}}{p-3}\right).
\end{equation}
\end{propo}



%
%
%


\vspace{1em}

The following theorem due to M. Rambaud in \cite{ramb1} generalizes essentially all the known formulas providing 
the current best symmetric upper-limit asymptotic bounds. 

\begin{theo}\label{TheoA-Sym}
Let $q$ a prime power and $r\geq 1$, $l\geq 1$ be two positive integers. Then, as long as the respective denominators are positive, we have:
\begin{enumerate}
\item[(a)] if $r=1$ and $q$ is such that $\betad_1(q)>5$
$$\Ms_q\leq \frac{2\mus_q(r,l)}{rl} \left(1+\frac{1}{rl\betad_r(q)-1}\right). $$
\item[(b)]
$$\Ms_q\leq \frac{2\mus_q(r,l)}{rl} \left(1+\frac{2}{rl\betad_r(q)-2}\right).$$
\item[(c)] if $2|q$
$$\Ms_q\leq \frac{2\mus_q(r,l)}{rl} \left(1+\frac{1+\log_q(2)}{rl\betad_r(q)-1-\log_q(2)}\right).$$
\item[(d)] if $2 \nmid q$
$$\Ms_q\leq \frac{2\mus_q(r,l)}{rl} \left(1+\frac{1+2\log_q(2)}{rl\betad_r(q)-1-2\log_q(2)}\right).$$
\end{enumerate}
\end{theo}

\begin{rema} 
In comparison to the other known results : 
\begin{itemize}
\item Bound (a) encompasses the upper-limit bounds of \ref{STV_m_repare} \ref{coro_STV_repare}, where it adds multiplicities of evaluation. This additional tool was introduced in \cite{arna1} and improved by \cite{ceoz}, then by \cite[Lemma 3.4]{randJComp}; 
\item Bound (b) allows evaluation on points of arbitrary degree compared to \cite[Proposition 11]{bachpi};
\item Bounds (c) and (d) allow evaluation on points of odd degree $r$ in \cite[Theorem 5.18]{cacrxi2}, and adds multiplicities of evaluation. Also, instead of using the formula $\betad_r(q)=(\sqrt{q^r}-1)/r$ in loc. cit., which is unproven in the general case, they are replaced here by $\betad_r(q)$. Notice that \emph{bounds (b) and (c) give stricly better numerical values than Proposition \ref{newbound2} for all values of $q$} for which Proposition \ref{newbound2} holds\footnote{Proposition \ref{newbound2} is let for the simplicity of its expression.}. Indeed, it suffices to use $r=2$ (and $l=1$), and to use the known value (\ref{DV2}) of $\betad_2(q)$ in Section \ref{CurveChoice}. 
\end{itemize}
\end{rema}

The following bounds are deduced from theorem \ref{TheoA-Sym}, except for $q=25$. We indicate the criterions (a) (b), etc. from which they are deduced, and the parameters $(r,l)$ used. The values $\betad_r(q)$ are directly taken from the known values given in Section \ref{courbes_Shimura}.


We detail how the upper bounds of the $\mus_q(r,l)$ are infered, because many where not directly published. Because of their interest, these bounds will be summarized in Section \ref{mus}.
To obtain these upper bounds we often use Formula (58) in \cite[Lemma 3.2]{randJComp} given by Inequality (\ref{IneqFormRecur}) in Section \ref{CurrentCCMA}:
\begin{equation} \label{eq:majomun} \mus_q(r,l)\leq \mus_{q^r}(1,l)\mus_{q}(r)\end{equation}
in particular
$$\mus_q(2,2)\leq \mus_{q^2}(1,2)\mus_{q}(2)\leq 3\times 3=9$$
(where the last two values are actually both equal to $3$, as shown by S. Winograd.

The biggest emphasis must be put on the following upper bound:
$$\mus_q(2,5)\leq 30$$
which is deduced from formula \eqref{eq:majomun} and from the upper bound:
\begin{equation} \label{eq:newmu4} \mus_4(1,5)\leq 10 \end{equation}
which was only published in \cite[Table 2]{rambWaifi}, in the justification of entry (1,10). It is regrettable that this record bound was not more emphasized in \cite{rambWaifi}: this has been repaired in \cite[Appendix §2.3]{ramb1}, where an explicit formula attaining this bound is given. Even more regrettable, the entry for (1,10) in the loc cit \cite[Table 1 \& Table 2]{rambWaifi} is grossly false. One should not read $\mus_q(1,10)\leq 30$ but instead $\mus_q(2,5)\leq 30$, as deduced from formula \eqref{eq:majomun} above. This was corrected in \cite[Table 3.1]{ramb1}. The error in \cite[Table 1 \& Table 2]{rambWaifi} comes from a grossly wrong application of formula \eqref{eq:majomun}.

Let us determine the values of the quantities $\mus_q(r,l)$ and $\mu_q(r,l)$ required in order to obtain Proposition \ref{newbound4}. 
All these values will be summarized in Sections \ref{mus} and \ref{mu}.

For $q=2$: from (b) with $(r,l)=(2,5)$ with $\mus_q(2,5)\leq 30$ as emphasized above.

For $q=3$:  (b) $(r,l)=(2,3)$ with $$\mu_3(2,3)\leq \mus_9(1,3)\mus_3(2,1)\leq \mus_3(1,3)\mus_3(2,1)\leq 5\times 3=15$$ where the 
latter, $3$, is from Karatsuba and the former, $5$, from \cite[Table 1 col. (2.4)]{ceozMod} (note that $5$ is actually \emph{equal} 
to the asymmetric complexity, by \cite[Table 3]{badeeszi}).

For $q=4$: (c) $(r,l)=(2,2)$ with $\mu_4(2,2)\leq 8 $ from \cite[(88)]{randJComp} (which, as a side remark, we even claim to be an equality, as follows from an unpublished exhaustive search performed while working on \cite[§1]{rambWaifi}). 

For $q=5$: (d) $(r,l)=(2,2)$ with $\mu_5(2,2)\leq 8 $ (\cite[(88)]{randJComp}).

For $q=7$: (d) $(r,l)=(2,1)$\footnote{Let us recall that $\mu_q(2,1)=\mu_q(2)=3$.}. 

For $q=8$: (c) $(r,l)=(2,1)$.

For $q=9$:  (d) $(r,l)=(2,1)$.

For $q=11$: (d) $(r,l)=(2,1)$.

For $q=25$  apply Proposition \ref{newbound}  obtained in \cite[Proposition 2]{bachpi}. 
\footnote{Notice that the authors did not apply themselves their bound to $q=25$, because it gives a higher value than the one from \cite{cacrxi2}: they did not know at the time that this latter bound was not actually proved. Note also that this bound is obtained by using the criterium $1)$ in \ref{theoprinc} with $a=0$, obtained in \cite[Theorem 1.1]{ball1}.}. 

\begin{propo}\label{newbound4}
$$\Ms_{2}\leq 10,$$
$$\Ms_3\leq 7.5,$$
$$\Ms_4\leq 5.33,$$
$$\Ms_{5}\leq 5.21,$$
$$\Ms_{7}\leq 4.08,$$
$$\Ms_{8}\leq 3.71,$$
$$\Ms_{9}\leq 3.77,$$
$$\Ms_{11}\leq 3.56,$$
$$\Ms_{25}\leq 3.$$
\end{propo}

These previous asymptotic bounds are the best published current numerical ones in the symmetric case\footnote{\samepage 
These bounds improve the following bounds: $\Ms_{2}\leq \frac{1035}{68}\simeq 15.23$ and 
$\Ms_3\leq \frac{1933}{250}\simeq 7.74$, obtained for $q=2$ and for $q=3$ in \cite[Theorem 4.9]{bapi2} (cf. also \cite[Theorem 4.9]{bapi2bis}) and for $q=4$  in \cite[Theorem 1.6 (i)]{bapirasi}: 
$\Ms_{4}\leq \frac{237}{39}\simeq 6.08$, which already improved the old following results : 
$\Ms_{2}\leq \frac{477}{26}\simeq 18.35$ obtained in \cite[Theorem 4.1]{bapi} 
and the old result $\Ms_{3}\leq 27$ obtained from \cite[Remark of Corollary 3.1]{ball3}.}. 

Now, if equation \eqref{eq:thB} did hold: $\betad_6(3)=\frac{3^3-1}{6}=13/3$, as would be implied e.g. by Conjecture \ref{conj:recmodp}, then applying criterion (b) to (6,1), using $\mus_3(6,1)\leq 15$ from \cite[table 1]{ceoz}, would yield $\Ms_3\leq \dfrac{65}{12}\simeq 5.41$. 
 And likewise for the couple of other bounds mentionned in \cite[Table 2.2]{ramb1} on the two lines named "Adding theorem B". Similarly, conjectures \ref{conjX}, \ref{conjY} and \ref{conjZ} would imply the bounds on the corresponding lines of \cite[Table 2.2]{ramb1}.

Then, using the general quantities linked to the $2$-torsion (cf. Section \ref{Bound2torsion}), I. Cascudo, R. Cramer, and C. Xing in\cite[Theorem 6.27]{cacrxi2} (cf. also \cite{cacrxi}) obtain the following general result:

\begin{theo}
Let $\F_q$ be a finite field. 
If there exists a real number ${a \leq A(q)}$ with ${a \geq 1+J_2(q,a)}$, then
$$m_q^{sym} \leq 2\left(1+\frac{1}{a-J_2(q,a)-1}\right).$$
In particular, if $A(q) \geq 1-J_2(a,A(q))$, then
$$m_q^{sym} \leq 2\left(1+\frac{1}{A(q)-J_2(q,A(q))-1}\right).$$
\end{theo}
Actually, Cascudo, Cramer and Xing stated their result in terms of $m_q$,
not of $m_q^{sym}$(cf. footnote \ref{symornot} Section \ref{mM}).
Here we stated it in terms of $m_q^{sym}$ because, as already explained,
the $2$-torsion really enters the play only when we restrict to symmetric algorithms.

In order to be useful, this result should be combined with upper bounds on the torsion-limit.
Some upper-bounds of this sort can be easily deduced from Weil's classical results on the torsion in Abelian varieties.
However, Cascudo, Cramer and Xing obtain a spectacular improvement using the Deuring-Shafarevich theorem.
This allows them to give an upper-bound on the $2$-torsion-limit of certain explicit towers (such as the Garcia-Stichtenoth tower),
as well as the following general result \cite[Theorem 2.3(iii)]{cacrxi2}:

\begin{theo}
Let $q=p^{2t}$ be an even power of a prime $p$.
Then we have
$$J_p(q,\sqrt{q}-1)\leq\frac{1}{(\sqrt{q}+1)\log_p(q)}.$$
\end{theo}

Despite this important progress, at this time this approach does not allow to obtain the claimed bounds 
by Shparlinski-Tsfasman-Vladut bound for symmetric complexity.
Indeed, for this, one has to show that the $2$-torsion-limit is $0$,
or equivalently, that $\delta_2^-(q)=0$ which is the open problem \ref{optorsionlimit}. 


Note that all the upper bounds on $M_q^{sym}$ obtained by I. Cascudo et al in \cite{cacrxiya} and \cite{cacrxi2} are unproved 
because the proofs are based on \cite[Lemma IV]{cacrxiya} which is not completely correct as it is shown in \cite[Section 3]{bapirasi} (cf. also  \cite{ramb1}). However, the bounds are correct under Conjecture \ref{conjY}
\footnote{\label{URcacrxiya}
The following results rely on the above unproven assumption: Theorem IV.6, Theorem IV.7 and the list of specific bounds in Corollary IV.8 of \cite{cacrxiya}. 
Also, Theorem 5.18 and the list of bounds in Corollary 5.19 of \cite{cacrxi}. More precisely, here is the unproved bounds:

\begin{itemize}
\item the symmetric bounds in Theorem IV.6, Theorem IV.7 and the list of specific bounds in Corollary IV.8 of \cite{cacrxiya}; namely the following:
$$\Ms_q \leq \mus_q(2t)\frac{q^t-1}{t(q^t-5)}$$
for any ${t \geq 1}$ as long as ${q^t-5> 0}$ for $q$ a prime power;
$$\Ms_q \leq \mus_q(t)\frac{q^{t/2}-1}{t(q^{t/2}-5)}$$
for any ${t \geq 1}$ as long as ${q^{t/2}-5> 0}$ for $q$ a prime power which is a square.

$$
\begin{array}{|c||c|c|c|c|c|c|c|c|c|}
	\hline
	q & 2 & 3 & 4 & 5 & 7 & 8 & 9 & 11 & 13 \\
	 \hline
 	 \Ms_q & 7.47 & 5.49  & 4.98 & 4.8 & 3.82 & 3.74 & 3.68 & 3.62 & 3.59 \\

	\hline
	\end{array}
	$$

\item also, the symmetric bounds in Theorem V.18 and the list of bounds in Corollary V.19 of \cite{cacrxi2}, namely:\\

$$\Ms_q \leq
\left\{
\begin{array}{ll}
\mus_q(2t)\frac{q^t-1}{t(q^t-2-\log_q2)} & \mbox{if } 2 | q \\
\mus_q(2t)\frac{q^t-1}{t(q^t-2-2\log_q2)} & \mbox{otherwise}
\end{array}
\right
.$$
for a prime power $q$ and for any $t \geq 1$ as long as $q^t-2-\log_q2 > 0$ for even $q$;
and $q^t-2-2\log_q2 > 0$ for odd $q$.

$$
\begin{array}{|c||c|c|c|c|}
	\hline
	q & 2 & 3 & 4 & 5  \\
	 \hline
 	 \Ms_q & 7.23 & 5.45  & 4.44 & 4.34  \\

	\hline
	\end{array}
	$$
\end{itemize}
}.

\section{Uniform bounds} \label{UniBounds}

\subsection{Some exact values for $\mus_{q}(n)$} \label{musexact}

Recall that by Theorem \ref{thm_wdg}, we have 
$\mus_q(n)=\mu_q(n)=2n-1$ if and only if $n\leq \dfrac{q}{2}+1$.
Applying CCMA with
well fitted elliptic curves, Shokrollahi in \cite{shok} (for the strict inequality) and Chaumine in \cite{chau1} have shown 
that:

\begin{theo}\label{thm_shokr}
If
\begin{equation}\label{ine}
\frac{1}{2}q +1< n \leq \frac{1}{2}(q+1+{\epsilon (q) })
\end{equation} 
where $\epsilon$ is the function defined by:
$$
\epsilon (q)=
\left \{
\begin{array}{l}
 \mbox{the greatest integer } \leq 2{\sqrt q} \mbox{ prime to $q$, if  $q$ is not a perfect square} \\
 2{\sqrt q},\mbox{ if $q$ is a perfect square,}
\end{array}
\right .
$$ 
then the symmetric bilinear complexity $\mus_q(n)$ of 
the multiplication in the finite extension $\F_{q^n}$ of the
finite field $\F_q$ is equal to $2n$. 
In particular, in this case, we have: $$\mus_q(n)=\mu_q(n).$$
\end{theo}

\begin{op}\label{reciproqueexactcomplexity}
We still do not know if the converse is true. More precisely the question is:
suppose that $\mu_q(n)=2n$, are the inequalities (\ref{ine}) true?
\end{op}

Moreover, for the values of $n$ not concerned by Theorems \ref{thm_wdg} and \ref{thm_shokr}, very few particular exact values 
are known and are all obtained in \cite{chch}:

\begin{table}
$$
\begin{array}{||c|c|c|c||}
\hline 
\hline
q & n & \mus_q(n) & \mu_q(n) \\
\hline
2 & 4 & 9 & 9\\
\hline
2 & 6 & 15 & 15\\
\hline 
\hline
\end{array}
$$
\caption{Exact bilinear complexities}\label{ExactBilinComp}
\end{table}

\begin{rema}
The bilinear complexity $\mu_2(4)=9$ is obtained in \cite[Example 3.2]{chch} by a personal computer program. It is easy to check 
this value can be obtained by a symmetric tensor corresponding to the iteration of the Karatsuba algorithm. 
Then $\mu_2(4)=\mus_2(4)=9$. The bilinear complexity $\mu_2(6)=15$ is obtained in \cite[Example 3.3]{chch} thanks to Inequality 
(\ref{plongement}) of Lemma \ref{lemasyMqmq} and a lower bound over the length of binary codes of dimension $6$ equal to the minimal distance. 

\end{rema}
 
 \begin{op}
Find exact values for $\mus_q(n)$ and $\mu_q(n)$. Find examples where $\mu_q(n)<\mus_q(n)$.
\end{op}

\subsection{Upper bounds for $\mus_{q}(n)$ and $\mus_{q}(l,r)$} \label{mus}

From the results of \cite{ball1} and the
algorithm of Corollary \ref{theo_deg12evalder} with ${\ell_1=\ell_2=0}$, we obtain (cf. \cite{ball1}, \cite{baro1}):

\begin{theo} \label{theoprinc}
Let $q$ be a prime power and let $n$ be an integer $>1$. 
Let $F/\F_q$ be an algebraic function field of genus $g$ 
and $N_k$ a number of places of degree $k$ in $F/\F_q$.
If $F/\F _q$ is such that there exists a place of degree $n$(which is always the case if 
$2g+1 \leq q^{\frac{n-1}{2}}(q^{\frac{1}{2}}-1)$) then:
\begin{enumerate}[1)]
	\item if $N_1 +a> 2n+2g-2$ for some integer $a\geq 0$, then $$ \mus_q(n) \leq 2n+g-1+a,$$
	\item if there exists a non-special divisor of degree $g-1$ (which is always the case if $q\geq 4$)
and $N_1+a_1+2(N_2+a_2)>2n+2g-2$ for some integers $a_1\geq 0$ and $a_2\geq 0$, 
then $$\mus_q(n)\leq 3n+2g+\frac{a_1}{2}+3a_2-1,$$
	\item if $N_1+2N_2>2n+4g-2$, then $$\mus_q(n)\leq 3n+6g.$$
\end{enumerate}
\end{theo}

\begin{rema}
The previous theorem enables to obtain general bounds on the bilinear complexity of the multiplication in $\F_{q^n}$ sur $\F_q$ 
from infinite families of algebraic function fields defined over $\F_q$. But a fixed finite field $\F_q^n$, if we want to obtain the 
best possible bound, we can search the best algebraic function field defined over $\F_q$ (i.e with the possible smallest genus) 
satisfying the conditions of this theorem.
\end{rema}

Finally, from good towers of algebraic functions fields satisfying Theorem \ref{theoprinc}, different improvements of the bounds of 
the symmetric bilinear complexity were successively obtained in \cite{ball1}, \cite{ball3}, \cite{baro1}, \cite{balbro}, \cite{ball4}, \cite{bach}, \cite{arna1}, \cite{bapi2}, and \cite{bapirasi}:

\begin{theo}\label{Cq}
Let $q=p^r$ be a power of the prime $p$ and let $n$ be an integer $>1$.
Then the symmetric bilinear complexity  of multiplication in any finite field $\F_{q^n}$  
is linear with respect to the extension degree $n$; more precisely, there exists a constant $C^{sym}_q$ such that for any $n>1$:
$$\mus_q(n) \leq C^{sym}_q n.$$ The best current values of the constants $C^{sym}_q$ are : 
$$
C^{sym}_q=
\left \{
\begin{array}{ll}

\mbox{if } q=2, & \mbox{then (1)} \quad 15.4575\\  
&  \mbox{see \cite[Corollary 29]{bapi2}}\\

\mbox{else if } q=3, &  \mbox{then (2)} \quad \dfrac{1933}{250}\simeq 7.732\\
&  \mbox{see \cite{bapi2} }\\
               
\mbox{else if } q=p \geq 7, &  \mbox{then (3)} \quad  3\left(1+ \frac{8}{3p-5}\right)\\
& \mbox{see \cite[Theorem 1.6 (ii)]{bapirasi}}\\

\mbox{else if } q=p^2 \geq 25, & \mbox{then (4)} \quad  2\left(1+\frac{2}{p-\frac{33}{16}}\right)\\
&  \mbox{see \cite[Theorem 1.7 (ii)]{bapirasi}} \\

\mbox{else if } q=p^{2k} \geq 64 \quad (k \geq 2), & \mbox{then (5)} \quad 2\left(1+\frac{p}{\sqrt{q}-3 + (p-1)\frac{\sqrt{q}}{\sqrt{q}+1}}\right)\\
&  \mbox{see \cite{arna1} and \cite[Theorem 1.7 (i)]{bapirasi}} \\

\mbox{else if } q \geq 4, & \mbox{then (6)} \quad 3\left(1+\frac{\frac{4}{3}p}{q-3+2(p-1)\frac{q}{q+1}}\right)\\
& \mbox{see \cite[Theorem 1.6 (i)]{bapirasi}}\\

\end{array}
\right .
$$
\end{theo}

\begin{rema}
Note that, from Corollary \ref{theo_deg12evalder} applied on a Garcia-Stichtenoth tower, 
N. Arnaud obtained in \cite{arna1} which is not published the bound (5) of Theorem \ref{Cq}.
In \cite{bapirasi}, the authors give a detailed proof of Bound (5). 
In \cite{bapirasi}, it is also proved the two revised bounds (3) and (4) for $\mu_{p^2}(n)$ 
and $\mu_p(n)$\footnote{\samepage In \cite{arna1}, N. Arnaud gives the two following bounds with no detailed calculation:
\begin{enumerate}
       \item[(3')]  If $p\geq5$ is a prime, then $\displaystyle{\mus_p(n)\leq 3 \left(1+ \frac{4}{p-1}\right) n}$.	
	\\
	\item[(4')] If $p\geq5$ is a prime, then  $\displaystyle{\mus_{p^2}(n) \leq 2 \left(1 + \frac{2}{p-2} \right)n}$.\\
\end{enumerate}
In fact, one can check that the denominators $p-1$ and $p-2$ are slightly overestimated under Arnaud's hypotheses.
}.
\end{rema}

Note also that the upper bounds\footnote{\samepage In \cite{ball5} and \cite{ball6}, S. Ballet gives the unproved following bounds: 
\begin{enumerate}
       \item[(1)]  If $q\geq 3$ is a prime power, then $\displaystyle{\mus_{q^2}(n)\leq 2 \left(1+ \frac{2}{q-2}\right) n}$,\\
	\item[(2)] If $q\geq5$ is a prime power, then  $\displaystyle{\mus_{q}(n) \leq 6 \left(1 + \frac{2}{q-2} \right)n}$,\\
	\item[(3)]  If $q=p^r> 3$ is a prime power, then $\displaystyle{\mus_{q}(n)\leq 3 \left(1+ \frac{2}{p-2}\right) n}$,\\
	\item[(4)] If $p>5$ is a prime, then  $\displaystyle{\mus_{q}(n) \leq 3 \left(1 + \frac{2}{p-2} \right)n}$.\\
	
\end{enumerate}
} 
obtained successively in \cite{ball5} and \cite{ball6} 
are obtained by using the mistaken statements of I. Shparlinski, M. Tsfasman and S. Vladut \cite{shtsvl} mentioned 
in the above section \ref{mM}. 

Moreover, for certain finite fields (in particular the cases of $\F_2$, $\F_3$ and $\F_4$), 
we have certain refined bounds for certain extensions obtained in \cite[Table 1]{ceoz}. 
Let us recall this table: 

 \vspace{.5em}

\hspace{-1.5em}{\setlength{\tabcolsep}{3.5pt}
\begin{table}
 \begin{tabular}{|l|*{17}{c|}}
	\hline
	$n$  &\  2 \  & \ 3 \ & \  4 \ & 5 & 6 & 7 & 8 & 9 & 10 & 11 & 12 & 13 & 14 & 15 & 16 & 17 & 18 \\
	 \hline
 	 $\mus_2(n) \leq$ & 3 & 6 & 9 & 13 & 15 & 22 & 24 & 30 & 33 & 39 & 42 & 48 & 51 & 54 & 60 & 67 & 69 \\
	\hline
	 $ \mus_3(n) \leq$ & 3 & 6 & 9 & 11 & 15 & 19 & 21 & 26 & 27 & 34 & 36 & 42 & 45 & 50 & 54 & 58 & 62 \\
	\hline
	$ \mus_4(n) \leq$ & 3 & 6 & 8 & 11 & 14 & 17 & 20 & 23 & 27 & 30 & 33 & 37 & 39 & 45 & 45 & 53 & 51 \\
	\hline
\end{tabular}

\vspace{.5em}

\caption{Best known bounds on complexities for small fields}\label{TableCenkOzbudak}
\end{table}
}

\vspace{.5em}

Moreover, in \cite[Tables 3 and 4]{babotu},  improving results obtained in \cite{ceoz} and \cite[Example 4.7]{randJComp}, 
bounds are given for certain particular extensions:

\vspace{.5em}

$$\begin{array}{|c|c|c|c|c|c|}
\hline
n & 163 & 233 & 283 & 409 & 571 \\
\hline
\mus_2(n) & 906 & 1340 & 1668 & 2495 & 3566 \\
\hline
\end{array}$$

\vspace{.5em}

$$\begin{array}{|c|c|c|c|c|c|}
\hline
n & 57 & 97 & 150 & 200 & 400 \\
\hline
\mus_3(n) & 234 & 410 & 643 & 878 & 1879 \\
\hline
\end{array}$$

\vspace{.5em}

The bounds presented in the previous tables are the best published current bounds for $\mu_q(n)$.
For the quantity $\mu_q(r,l)$, with $l>1$, different values have been given by M. Rambaud in \cite{ramb1} and 
explained in Section \ref{msMs}. Let us summarize for $q=2$ these values (including the case l=1) in the following table \ref{UBCmuqrl}.

\begin{table}[h]
$$
\begin{array}{ccc}
\begin{array}{|c||c|c|c|c|}
\hline
l \backslash r & 1& 2 & 3 & 4\\
\hline
\hline
1 & 1 & 3 & 6 & 9 \\
\hline
2 & 3 & 9 & 16 & 24 \\
\hline
3 & 5 & 15& 30 & \\
\hline
4 & 8& 21& & \\
\hline
5 & 11 & 30 & & \\
\hline
\end{array}
& \hbox{      } & 
\begin{array}{|c||c|c|c|c|}
\hline
l \backslash r & 1& 2 & 3 & 4\\
\hline
\hline
6 & 14& & & \\
\hline
7 & 18 & & & \\
\hline
8 & 22 & & & \\
\hline
9 & 27 & & & \\
\hline
10 & 31 & & & \\
\hline
\end{array}
\end{array}
$$
\caption{Upper bounds on the complexities $\mus_2(r,l)$.}\label{UBCmuqrl}
\end{table}

For other values of $q$ let us summarize the known results, obtained in Section  \ref{msMs}.

$$
\mus_q(2,2)\leq 9, \hbox{ } \mus_q(2,5)\leq 30.
$$

$$
\mus_4(1,5) \leq 10.
$$


\medskip

Recently in \cite{bazy2}, S. Ballet and A. Zykin would improve all the known uniform 
upper bounds for $\mus_{p^2}(n)$ and $\mus_{p}(n)$ for a prime $p\geq 5$.
Their approach consists on using dense families of modular curves which are not obtained asymptotically 
thanks to prime number density theorems 
of type Hoheisel, in particular a result due to Dudek \cite{dude}. Note that one of main ideas used in \cite{bazy2} was 
introduced in \cite{ball5} by S. Ballet thanks to the use of the Chebyshev Theorem (or also called the Bertrand Postulat) to 
bound the gaps between prime numbers in order to construct families of modular 
curves as dense as possible. Later, motivated by \cite{ball5}, 
the approach of using such bounds on gaps between prime numbers 
(e.g. Baker-Harman-Pintz \cite{bahapi}) was also used by H. Randriambololona in the 
preprint \cite{rand2D-G} in order to improve the upper bounds of $\mus_{p^2}(n)$ where $p$ is a prime number. 
In summary, let us give the new uniform bounds given there (and recalled in \cite{randGap2}).

In order to present these bounds, let us recall the following notation.
For any infinite subset ${\mathcal A}$ of $\N$ and for any real $x>0$, 
let $$\lceil x \rceil_{{\mathcal A}}=\min {\mathcal A}\cap [x,+\infty [$$ 
be the smallest element of ${\mathcal A}$ larger than or equal to $x$.
Also set: 
$$\epsilon_{{\mathcal A}}(x)=\sup_{y\geq x}\frac{\lceil y \rceil_{{\mathcal A}}-y}  {y}.$$

Now, we have:


\begin{propo}\label{newboundpcarréfusion}
Let $p\geq 7$ be a prime number. Then:
\begin{enumerate}
\item for all $k\geq \frac{p^2+p+1}{2}$, $$\frac{1}{k}\mus_{p^2}(k)\leq 2\biggr(1+\frac{1+\epsilon_{{\mathcal P}}(\frac{24k}{p-2})}{p-2} \biggl).$$
\item for all $k\geq 1$, $$\frac{1}{k}\mus_{p^2}(k)\leq 2\biggr(1+\frac{2}{p-2} \biggl).$$
\item for all $k\geq 1$, $$\frac{1}{k}\mus_{p^2}(k)\leq 2\biggr(1+\frac{1+\frac{10}{139}}{p-2} \biggl)$$
\item for all $k\geq e^{50}p$, $$\frac{1}{k}\mus_{p^2}(k)\leq 2\biggr(1+\frac{1.000000005}{p-2} \biggl)$$
\item for all $k\geq 16531(p-2)$, $$\frac{1}{k}\mus_{p^2}(k)\leq 2\biggr(1+\frac{1+\frac{1}{25\log^2(\frac{24k}{p-2})}}{p-2} \biggl)$$
\item for $k$ large enough, $$\frac{1}{k}\mus_{p^2}(k)\leq 2\biggr(1+\frac{1+\frac{1}{(\frac{24k}{p-2})^{0.475}}}{p-2} \biggl).$$


\end{enumerate}

\end{propo}

Recently, combining his results of \cite{rand2D-G} with the result of A. Dudek \cite{dude} as in \cite{bazy2}, 
H. Randriambolona improves in \cite{randGap2} almost all these bounds except for the case $q=p^2=25$ obtained 
in \cite{bazy2}. In summary, let us give the new uniform bound of the symmetric bilinear complexity 
given respectively in \cite[Corollary 10]{randGap2} and \cite[Proposition 7]{bazy2}. 

\begin{propo}\label{newboundpcarréfusion2}
Let $p\geq 7$ be a prime number. Then:
\begin{enumerate}
\item[(7)] for all $k\geq \frac{p-2}{24}e^{e^{33.217}}$, 
$$\frac{1}{k}\mus_{p^2}(k)\leq 2\biggr(1+\frac{1+\frac{3}{(\frac{24k}{p-2})^{\frac{1}{3}}}}{p-2} \biggl).$$
\end{enumerate}

\end{propo}

\begin{propo}\label{newboundpcarré}
Let $x_{\alpha}$ be the constant defined in  \cite[Theorem 6]{bazy2} (recalled in Theorem \ref{lemmek0}).
For any integer $ n\geq x_\alpha+3$ we have

$$
\mus_{25}(n) \leq 2\left(1+\frac{1+n^{\alpha-1})}{2}\right)n-3n^{\alpha-1}- 4.
$$
\end{propo}

Let us recall the following key result as direct consequence of the results of 
Baker, Harman, and Pintz \cite{bahapi} and A. Dudek \cite{dude} 
on which Assertion (vi) in Proposition \ref{newboundpcarréfusion}, 
Proposition \ref{newboundpcarréfusion2} as well as Proposition \ref{newboundpcarré} are essentially based on.

Their results concern explicit prime number density theorems, usually called theorems of type Hoheisel. 
In particular, by a result of Baker, Harman and Pintz \cite[Theorem 1]{bahapi} established in 2001 
and by a recent result established by Dudek  \cite[Theorem 1.1]{dude} in 2016, we directly deduce 
the following result \cite[Theorem 6]{bazy2}:

\vspace{.5em}

\begin{theo}\label{lemmek0}
Let $l_k$ be the $k$-th prime number. Then there exist real numbers $\alpha<1$ and $x_{\alpha}$ such 
that the difference between two consecutive prime numbers $l_k$ and $l_{k+1}$ satisfies 
$$l_{k+1}-l_k\leq l_k^{\alpha}$$ for any prime $l_k\geq x_{\alpha}.$ 
In particular, one can take $\alpha=\frac{2}{3}$ with $x_{\alpha}=\exp(\exp(33.217))$.
Moreover, one could take $\alpha=\frac{21}{40}$ with a value of $x_{\alpha}$ that could in principle be determined effectively.
\end{theo}


%

\begin{op}
A problem which is highly not trivial consists on determining effectively a value of 
$x_{\alpha}$ for $\alpha=\frac{21}{40}$. This problem is a typical problem of analytic number 
theory, said problem of type Hoheisel.
\end{op}


Then, the second result concerns the case of prime fields. The optimal method used 
by H. Randriambolona \cite{randGap2} for solving Riemann-Roch systems (cf. Section \ref{Bound2torsion}) 
does not work well for symmetric algorithms over prime fields. Instead, to prove 
\cite[Proposition 10]{bazy2} Ballet and Zykin use a suboptimal method from \cite{baro4} 
associated to descent technics (cf. Section \ref{descent}) and obtain:

\begin{propo}\label{newboundpbz}
Let $p\geq 5$ be a prime number, let $x_{\alpha}$ be defined as in Theorem \ref{lemmek0}.
\begin{enumerate}
\item If $p\neq 11,$ then for any integer $ n\geq \frac{p-3}{2}x_\alpha+\frac{p+1}{2}$ we have

$$
\mus_{p}(n) \leq 3\left(1+\frac{\frac{4}{3}(1+\epsilon_p(n))}{p-3}\right)n-\frac{2(1+\epsilon_p(n))(p+1)}{p-3},
$$
where $\epsilon_p(n)=\left(\frac{2n}{p-3}\right)^{\alpha-1}.$

\item For $p=11$ and $ n\geq (p-3)x_\alpha+p-1=8x_\alpha+10$ we have
$$
\mus_{p}(n) \leq 3\left(1+\frac{\frac{4}{3}(1+\epsilon_p(n))}{p-3}\right)n-\frac{4(1+\epsilon_p(n))(p-1)}{p-3}+1,
$$
where $\epsilon_p(n)=\left(\frac{2n}{p-3}\right)^{\alpha-1}.$
\end{enumerate}
\end{propo}

\subsection{Upper bounds for $\mu_{q}(n)$ and $\mu_{q}(l,r)$}\label{mu}

By using the asymmetric part of Theorem~\ref{theo_evalder}, J.~Pieltant and
H.~Randriambololona obtained in \cite{pira} results about bilinear complexity not necessarily symmetric.
In particular, they obtain the best bounds in the extensions of $\F_2$, $\F_p$ and $\F_{p^2}$
for all $p \geq 3$ and $\F_q$ and $\F_{q^2}$ for all $q \geq 4$.

\begin{propo}
Let $q$ be a prime power and $d$ be an positive integer for which all proper divisors
verify $j < \frac{1}{2}(q+1+\epsilon(q))$ if $q \geq 4$, or $j \leq \frac{1}{2}q+1$ if $q \in \{2,3\}$.
Let $F/\F_q$ be an algebraic function field of genus $g \geq 2$ with $N_i$ places
of degree $i$ and let $\ell_i$ be integers such that $0 \leq \ell_i \leq N_i$, for all $i | d$.
Suppose that:
\begin{enumerate}[(i)]
   \item there exists a place of degree $n$ in $F/\F_q$,
   \item $\sum_{i|d} i(N_i+\ell_i) \geq 2n+g+\alpha_q$,
   where $\alpha_2=5, \alpha_3=\alpha_4=\alpha_5=2$ and $\alpha_q=-1$ for $q > 5$.
\end{enumerate}
Then
$$
\mu_q(n) \leq \frac{2\mus_q(d)}{d}\left(n+\frac{g}{2}\right)+\gamma_{q,d}\sum_{i|d} i\ell_i+\kappa_{q,d},
$$
where $\gamma_{q,d}=\max_{i|d}\left(\frac{\mu_q(i,2)}{i}\right)-\frac{2\mus_q(d)}{d}$
and $\kappa_{q,d} \leq \frac{\mus_q(d)}{d}(\alpha_q+d-1)$.
\end{propo}

By choosing $d=1,2$ or $4$, they obtain the two following corollaries:

\begin{coro}
Let $q \geq 3$ be a prime power and $F/\F_q$ be an algebraic function field of genus $q \geq 2$ with $N_i$ places
of degree $i$.
Let $\ell_i$ be integers such that $0 \leq \ell_i \leq N_i$.
Suppose that:
\begin{enumerate}[(i)]
   \item there is a place of degree $n$ in $F/\F_q$,
   \item $N_1+\ell_1+2(N_2+\ell_2) \geq 2n+g+\alpha_q$,
   where $\alpha_3=\alpha_4=\alpha_5=2$ and $\alpha_q=-1$ for $q > 5$.
\end{enumerate}
Then
$$\mu_3(n) \leq 3n+\frac{3}{2}g+\frac{3}{2}(\ell_1+2\ell_2)+\frac{9}{2},$$
$$\mu_q(n) \leq 3n+\frac{3}{2}g+\ell_1+2\ell_2+\frac{9}{2},\mbox{ for } q=4 \mbox{ or } 5,$$
and for $q > 5$,
$$\mu_q(n) \leq 3n+\frac{3}{2}g+\frac{1}{2}(\ell_1+2\ell_2),$$
or in the particular case where $N_2=\ell_2=0$
$$\mu_q(n) \leq 2n+g+\ell_1-1.$$
\end{coro}

\begin{coro}
Let $F/\F_2$ be an algebraic function field of genus $g \geq 2$ with $N_i$ places of degree $i$
and let $\ell_i$ be integers such that $0 \leq \ell_i \leq N_i$.
Suppose that:
\begin{enumerate}[(i)]
   \item there is a place of degree $n$ in $F/\F_2$,
   \item $\sum_{i|4}i(N_i+\ell_i) \geq 2n+g+5$,
\end{enumerate}
then
$$\mu_2(n) \leq \frac{9}{2}\left(n+\frac{g}{2}\right)+\frac{3}{2}\sum_{i|4}i\ell_i+18.$$
\end{coro}

Then, they establish new asymmetrical uniform bounds:

\begin{theo}\label{bornesasympt}
For $n \geq 2$,
\begin{enumerate}[(i)]
   \item if $q=2$, then
   $$\mu_2(n) \leq \frac{189}{22}n + 18,$$
   \item if $q=3$, then
   $$\mu_3(n) \leq 6n,$$
   \item if $q=4$, then
   $$\mu_4(n) \leq \frac{87}{19}n,$$
   \item if $q=5$, then
   $$\mu_5(n) \leq \frac{9}{2}n,$$
   \item if $q \geq 4$, then
   $$\mu_{q^2}(n) \leq 2\left(1+\frac{p}{q-2+(p-1)\frac{q}{q+1}}\right)n-1,$$
   \item if $p \geq 3$, then
   $$\mu_{p^2}(n) \leq 2\left(1+ \frac{2}{p-1}\right)n-1,$$
   \item if $q>5$, then
   $$\mu_q(n) \leq 3\left(1+\frac{p}{q-2+(p-1)\frac{q}{q+1}}\right)n,$$
   \item if $p>5$, then
   $$\mu_p(n) \leq 3\left(1+\frac{2}{p-1}\right)n.$$
\end{enumerate}
\end{theo}

Recently, by using the same dense families of modular curves defined over $\F_p$ than the one used to get 
Theorem \ref{newboundpcarréfusion} 
in Section \ref{mus}, H. Randriambololona obtains the following result.

\begin{propo}\label{newboundphr}
Let $p\geq 7$ be a prime number. Then:
\begin{enumerate}
\item for all $k> \frac{p+1}{2}$, $$\frac{1}{k}\mu_{p}(k)\leq 
3\biggr(1+\frac{1+\epsilon_{{\mathcal P}}(\frac{24k}{p-2})}{p-2} \biggl),$$
\item for all $k\geq \frac{p-2}{24}e^{e^{33.217}}$, 
$$\frac{1}{k}\mu_{p}(k)\leq 3\biggr(1+\frac{1+\frac{3}{(\frac{24k}{p-2})^{\frac{1}{3}}}}{p-2} \biggl),$$
\item for $k$ large enough, 
$$\frac{1}{k}\mu_{p}(k)\leq 3\biggr(1+\frac{1+\frac{1}{(\frac{24k}{p-2})^{0.475}}}{p-2} \biggl).$$
\end{enumerate}
\end{propo}

\begin{rema}
Note that the difficulty of solving the Riemann-Roch systems (cf. \ref{construction_sym}) in 
the context of symmetric algorithms using curves having not sufficiently rational points is 
avoided here, since the previous result is obtained by using the asymmetric version of type 
Chudnovsky algorithm (cf. Section Section  \ref{DiscussionSym} and Section \ref{CurrentCCMA}) applied over places of degree two.
\end{rema}

Now, let us recall some particular values of the quantities $\mu_q(l,r)$, obtained in Section \ref{msMs}:
$$
\mu_3(2,3) \leq 15, \hbox{ } \mu_4(2,2) \leq 8, \hbox{ } \mu_5(2,2) \leq 8.
$$

\section{Effective construction of bilinear multiplication algorithms}\label{SecEffective}

In this section, we are interested by the study of the effective construction of bilinear multiplication algorithms in finite fields. 
Little few work has been done 
on the effective construction of the algorithms of type Chudnovky. They are mainly contained in the following articles:
\cite{bash}, \cite{ball2}, \cite{ceoz}, \cite{babotu}, \cite{atbaboro1} and  \cite{atbaboro2}.

\subsection{Non-asymptotic construction}

\subsubsection{Classical multiplication algorithms}

\begin{itemize}
\item[a)] Example of an effective symmetric construction using an elliptic curve.

\vspace{.5em}

This example developped by U. Baum and A. Shokrollahi in \cite{bash} is the first effective construction 
of an bilinear algorithm of multiplication which implements CCMA. It concerns a multiplication algorithm in the finite field 
$\F_{256}$ over $\F_4$, namely $q=4$ and $n=4$, using the maximal Fermat elliptic curve $y^2+y=x^3+1$. 
The bilinear complexity $\mus({\mathcal U})$of this 
symmetric algorithm ${\mathcal U}$ is optimal and such that $$\mus({\mathcal U})=\mus_q(n)=\mu_q(n)=2n=8.$$

\vspace{.5em}

\item[b)] Example of effective symmetric constructions using an hyperelliptic curve. 

\vspace{.5em}

This example developped by S. Ballet in \cite{ball2} is the first effective construction 
of an bilinear algorithm of multiplication which implements CCMA for an algebraic curve of genus $g>1$.
It concerns a multiplication algorithm in the finite field 
$\F_{16^n}$ over $\F_{16}$, more precisely $q=16$ and $n=13, 14, 15$, using the maximal hyperelliptic curve $y^2+y=x^5$. 
The bilinear complexity of this symmetric algorithm ${\mathcal U}$ is quasi-optimal and such that 

$$\mus({\mathcal U})= 2n+1,$$ which proves that $2n\leq\mu_q(n)\leq\mus_q(n)\leq 2n+1$. 

\begin{op}
Find the exact bilinear complexity in these finite fields $\F_{16^n}$ over $\F_{16}$ with $n=13, 14, 15$, 
knowing that this complexity is $2n$ or $2n+1$. 
Optimize the scalar complexity of these constructions.
\end{op}

\vspace{.5em}

\item[c)] Example of an effective symmetric construction using higher degree places and derivated 
evaluations on rational places on elliptic curves. 

\vspace{.5em}

This example developped by M. Cenk and F. \"Ozbudak in \cite{ceoz} is the first effective construction 
of an bilinear algorithm of multiplication which implements the combination of the generalizations 
of CCMA introduced in \cite{baro1} using places of degree one and two and in \cite{arna1} 
using derivated evaluations. Note that in this example, the derivated evaluations are only used 
on rational places at the order one. 
More precisely,  it concerns a multiplication algorithm in the finite field $\F_{3^9}$ over $\F_{3}$ 
using the non-optimal elliptic curve $y^2=x^3+x+2$. 
In this case, the authors use the evaluation on four rational places with derivated evaluation on two 
among them as well as the evaluation on six places of degree two.
The bilinear complexity of this symmetric algorithm ${\mathcal U}$ is such that 

$$\mus({\mathcal U})=4+2\times 2+6\times 3=26.$$ 

\vspace{.5em}

\item[d)] Example of effective asymmetric construction using higher degree places on algebraic curves.

\vspace{.5em}

This example developped by  S. Ballet, N. Baudru, A. Bonnecaze and M. Tukumuli  
in \cite{bababotu2} (announced in \cite{bababotu}) and by Tukumuli in \cite{tuku} 
is the first effective construction of bilinear algorithms of multiplication which 
implements the asymetric generalization of CCMA introduced in \cite{randJComp}. 
Note that these examples use two distinct Riemann-Roch spaces ${\mathcal L}(D_1)$ 
and ${\mathcal L}(D_2)$ without derivated evaluations. 
More precisely,  in \cite{bababotu2}, three algorithms are constructed. 
The first example concerns a multiplication algorithm in the finite field $\F_{16^{13}}$ over $\F_{16}$ 
using the maximal hyperelliptic curve $y^2+y=x^5$ and only rational places on it. 
The second example concerns a multiplication algorithm in the finite field $\F_{4^{4}}$ over $\F_{4}$ using the optimal curve 
$y^2+y=\frac{x}{x^3+x+1}$ over $\F_4$. 
The third example concerns a multiplication algorithm in the finite field $\F_{2^{5}}$ over $\F_{2}$ using the optimal curve 
$y^2+y=\frac{x}{x^3+x+1}$ over $\F_4$.

\vspace{.5em}

\end{itemize}

\subsubsection{Parallel algorithms designed for multiplication and exponentiation}

In \cite{atbaboro1} and  \cite{atbaboro2}, thanks to a new construction of CCMA, 
 K. Atighechi, S. Ballet, A. Bonnecaze, and R. Rolland design efficient algorithms for both 
 the exponentiation and the multiplication in finite fields. 
They are tailored to hardware implementation and 
they allow computations to be parallelized while maintaining a low number of bilinear multiplications. 
 Notice that so far, practical implementations of multiplication 
 algorithms over finite fields have failed to simultaneously optimize the number of scalar multiplications, additions and bilinear multiplications.
 Regarding exponentiation algorithms, the use of a normal basis is of interest  because the $q^{th}$ power 
 of an element is just a cyclic shift of its coordinates. A remaining question is, how to implement 
 multiplication efficiently  in order to have simultaneously fast multiplication and fast exponentiation. In 2000,  S. Gao et al. \cite{gao} show that 
 fast multiplication methods can beadapted to normal bases constructed with Gauss periods. They show that if $\F_{q^n}$ is represented by a normal 
basis over $\F_q$ generated by a Gauss period of type $(n,k)$, the multiplication in $\F_{q^n}$ 
can be computed with $\BigO{nk\log nk \log\log nk}$ and the exponentiation with $\BigO{n^2k\log k \log\log nk}$ operations in $\F_q$ ($q$ being small). 
This result is valuable when $k$ is bounded. However, in the general case $k$ is upper-bounded by $\BigO{n^3\log^2 nq }$.
 
In 2009, J.-M. Couveignes and R. Lercier construct in \cite[Theorem 4]{cole} 
two families of basis (called elliptic and normal elliptic) for finite field extensions from which they obtain a 
model $\Xi$ defined as follows.
To every couple $(q,n)$, they associate a model, $\Xi(q,n)$, of the degree $n$ extension of $\F_q$ such that the following holds: 
there is a positive constant $K$ such that the following are true: 

- Elements in $\F_{q^n}$ are represented by vectors for which the number of components in $\F_q$ is upper bounded by  
$Kn(\log n)^2\log (\log n)^2.$

- There exists an algorithm that multiplies two elements at the expense of $Kn(\log n)^4|\log (\log n)|^3$ 
multiplications in $\F_q$.

- Exponentiation by $q$ consists in a circular shift of the coordinates.

\vspace{.5em}

Therefore, for each extension of finite field, they show that there exists a model which allows both fast 
multiplication and fast application of the Frobenius automorphism.
Their model has the advantage of existing for all extensions. 
However, the bilinear complexity of their algorithm is not competitive compared with the best known 
methods, as pointed out in \cite[Section 4.3.4]{cole}. Indeed, it is clear that such a model requires at least 
$Kn(\log n)^2(\log (\log n))^2$  bilinear multiplications.


\vspace{1em}

The authors of \cite{atbaboro2} propose another model with the following characteristics:

- The model is based on CCMA, thus the multiplication algorithm has a bilinear complexity in $O(n)$, which is optimal. 

- The model is tailored to parallel computation.  Hence, the computation time used to perform a multiplication or 
any exponentiation can easily be reduced with an adequate number of processors.  
Since the method has a bilinear complexity of multiplication in $O(n)$, it can be parallelized to obtain a constant 
time complexity using $\BigO{n}$ processors. The previous aforementioned works (\cite{gao} and \cite{cole}) 
do not give any parallel algorithm (such an algorithm is more difficult to conceive than a serial one). 

- Exponentiation by $q$ is a circular shift of the coordinates and can be considered free. Thus, efficient parallelization 
can be done  when doing exponentiation.

 - The scalar complexity of their exponentiation algorithm is reduced, compare to a basic 
 exponentiation using CCMA, thanks to a suitable basis representation of the Riemann-Roch space ${\mathcal L}(2D)$ in the second evaluation map.  
More precisely, the normal basis representation of the residue class field is carried in the associated Riemann-roch space ${\mathcal L}
(D)$, and the exponentiation by $q$ consists in a circular shift of the $n$ first coordinates of the vectors lying in the Riemann-Roch space ${\mathcal L}(2D)$.

 - The model uses Coppersmith-Winograd \cite{cowi1} method (denoted CW) or any variants 
 thereof to improve matrix products  and to diminish the number of scalar operations. 
 
\begin{op}
The structure of the involved matrices in the algorithm CCMA should be looked at more closely 
but unfortunately, there are no theoretical means or criteria today to build the best matrices because 
they depend on the geometry of the curves, the field of definition of these curves, as well as the Riemann-Roch 
spaces involved. A study of suitable optimisation strategies of CCMA from this point of view can be found in \cite{baboda}. In particular, the algorithm CCMA using an elliptic curve for multiplication in $\F_{256}/\F_4$ constructed by U. Baum and A. Shokrollahi \cite{bash} is improved.
The remaining open question is how to choose the geometrical objects in order to minimise the number of zeroes in a matrix 
of the evaluation map on the rational points of a curve. 
\end{op}

\subsection{Asymptotic construction}

D. V. and G.V. Chudnovsky claim in \cite{chch} that one can construct in 
polynomial time bilinear multiplication algorithm realizing 
a a bilinear complexity attaining the upper bound for $m_q$. 
Then, I. Shparlinsky, M. Tsfasman and S. Vladut in \cite{shtsvl} note that the argument 
of D. V. and G.V. Chudnovsky is insufficient. Indeed, the construction of such algorithms involves some random 
choice of divisors having prescribed properties 
over an exponentially large set of divisors.

\vspace{.5em}

I. Shparlinsky, M. Tsfasman and S. Vladut obtain a partial result concerning this polynomial 
construction by the following way. Let $q= p^{2m}\geq 49$ 
and let  $X_i=X_0(11l_i)$ be the reduction of the classical modular curve, 
$l_i$ being the i-th prime (for $q=p^2$), or $X_i=X_0(p_i)$ where $p_i$ is 
an irreducible polynomial over $\F_q$ of odd degree coprime with $q-1$ (for $q=p^{2m}$). 
Here, $X_0(p_i)$ is the reduction of the Drinfeld modular curve. 
Note that $\{X_i\}$ is a family of absolutely irreducible smooth curves of 
genus $g=g_i$ with $\lim_{g \rightarrow \infty} \frac{\mid X(\F_q)\mid}{g}={\sqrt q}-1$.  
Then, they prove the following result:

\begin{propo}
Suppose that for a family of modular curves described above for any $X\in \{X_i\}$ there is given 
an explicit point $Q$ of $X$ of some degree $n$ 
such that $$g.\left({\sqrt q}-5\right)/2-o(g)\leq n\leq g.\left({\sqrt q}-5\right)/2.$$
Let $Q$ be defined by its coordinates in some projective embeddings. Then one can polynomially construct a 
sequence ${\mathcal U}= {\mathcal U}_i$ of bilinear multiplication 
algorithms in finite fields $\F_{q^n}$ for the given sequence of 
$n\rightarrow \infty$ such that $$\lim_{g \rightarrow \infty}\mu^{sym}({\mathcal U})/n=2\left(1+\frac{4}{{\sqrt q}-5}\right).$$
\end{propo}

\vspace{.5em}

This proposition means that to get a polynomially constructable algorithm with linear complexity, one needs to 
construct explicitly (i.e polynomially) points of corresponding 
degrees on modular curves (or on other curves with many points). Unfortunately, so far 
it is unknown how to produce such points.

\vspace{.5em}

In \cite[Remark 6.6]{randJComp}, H. Randriambololona improves this result under the same hypothesis concerning the construction 
of a point of degree $n$. 
More precisely, up to this existence, he obtains a polynomial time (in $n$) construction of a multiplication algorithm (respect. 
a symmetric multiplication algorithm) 
in $\F_{q^n}/\F_q$ of length $2n(1+\frac{1}{\sqrt{q}-2})+o(n)$ for $q\geq 9$ (resp. $q\geq 49$).

\vspace{.5em}

In \cite{babotu}, S. Ballet, A. Bonnecaze and M. Tukumuli obtain a polynomial construction 
of a symmetric multiplication algorithm 
of type elliptic Chudnovsky--Chudnovsky (i.e with the Chudnovsky-Chudnovky interpolation 
method on an elliptic curve) of length in
$O(n(2q/K)^{\log^{\star}(n)})$ 
where 
\begin{equation}
\log^{\star}(n)=
\left \{
\begin{array}{ll}
0 & \hbox{if } n\leq 1,  \\
1+ \log^{\star}(\log n) & \hbox{otherwise,}   \\  
\end{array}\right .
\end{equation}
$K=2/3$ if the characteristic of $\F_q$ is $2$ or $3$ and $K=5/8$ otherwise.
Note that the length is only quasi-linear in $n$. However, this construction is without the restriction linked 
to the construction of a point of degree $n$.
Moreover, this asymptotical construction is not realized from an infinite family of 
suitable curves as the above results but thanks 
to the use of a sequence $\mathcal{A}_{q,n}$ of symmetric bilinear multiplication algorithms constructed from an 
arbitrary elliptic curve defined 
over $\F_q$ and using high degree points of this curve.

\vspace{.5em}

In \cite{bsho2}, N. Bshouty gives a deterministic polynomial time construction of a 
tester of type $(\mathcal{H}\mathcal{L}\mathcal{F}(\F_q,n,d), \F_{q^n},\F_q))$
and of size $\mu=O(d^{\tau(d,q)}n)$ where 
\begin{equation}
\tau(d,q)=
\left \{
\begin{array}{llll}
3 & \hbox{if} & q\geq cd^2,  & c>1 \hbox{ constant, } q \hbox{ perfect square,} \\
4 & \hbox{if} & q\geq cd,  & c>1 \hbox{ constant, }  \\
5 & \hbox{if} & q\geq d+1,  &   \\
6 & \hbox{if} & q=d.  &   \\
\end{array}\right .
\end{equation}

From \cite{bsho2}, in \cite[Corollary 2]{bsho}, N. Bshouty gives the first polynomial 
time construction of a multilinear multiplication algorithm 
with linear multiplicative complexity in $O(d^{\tau(d,q)}n)$ for the multiplication of $d$ 
elements in any extension finite field $\F_{q^n}$. 
This solves the open problem of deterministic polynomial time constructing a bilinear 
algorithm (i.e with $d=2$) with linear bilinear complexity for 
the multiplication of two elements in finite fields \cite{chch}\cite{shtsvl}\cite{ball4}. However, it does 
not solve the problem of deterministic polynomial time constructing 
a bilinear algorithm of type Chudnovsky--Chudnovsky. Indeed, the method of N. Bshouty is only based 
upon the equivalence between an optimal tester size 
and multilinear complexity. More precisely, the minimal size of a tester for 
$\mathcal{H}\mathcal{L}\mathcal{F}(\F_q,n,d)$ turns out to be equivalent to
the rank of the tensor of the multiplication of $d$ elements in $\F_{q^n}$ over $\F_q$.
The minimal size of a tester for $\mathcal{H}\mathcal{P}(\F_q,n,d)$ is equivalent to
the symmetric rank of the tensor of multiplication of $d$ elements.




 \section{Appendix: proof of Theorem \ref{TheoA-Sym}, Theorem \ref{TheoA-Asym} and Proposition \ref{prop:thAconjZ}} \label{SecAnnexe}
 We compress here the proof in \cite[II §1.2-3]{ramb1}.
\subsection{Repairing (and extending) the criterion of Cascudo \& al}
The following theorem does control for 2-torsion in the worst case. It is a straight generalization of \cite[Theorem 5.18]{cacrxi2}. The parameters will be later specified in the next paragraph to derive criterions for asymptotic bounds.

\begin{theo} \label{th:ccx}
Let $X$ be a curve of genus $g$ over $\F_q$, where $q\geq 2$ is any prime power, and let $m\geq 1$ be an integer.

Suppose that $X$ admits a closed point $Q$ of degree $\deg Q=m$ (a sufficient condition for this is $2g+1\leq q^{(m-1)/2}(q^{1/2}-1)$).

Consider now a collection of integers $n_{d,u}\geq 0$ (for $d,u\geq1$), such that almost all of them are zero, and that for any $d$,
\begin{equation} 
n_d=\sum_u n_{d,u}\leq B_d(X),
\end{equation}
where $B_d(X)$ denotes the number of closed points of $X$ of degree $d$.

Let $R$ the smallest \emph{integer} such that
\begin{align} \label{eq:critqeven} R\geq g(1+\log_q(2)) + 2m +3log_q\left(\frac{3qg}{(\sqrt{q}-1)^2} \right) +2 \text{ (if $2|q$)} \\
 R\geq g(1+2\log_q(2)) + 2m +3log_q\left(\frac{3qg}{(\sqrt{q}-1)^2} \right)+2  \text{ (otherwise).} \end{align}
Then, provided
\begin{equation*} \label{eq:thccx} \sum_{d,u} n_{d,u}du \geq R \end{equation*}
we have
\begin{equation} \mu_q(m)\leq\sum_{d,u}n_{d,u}\mu_q(d,u). \end{equation}
\end{theo}


The following proposition gathers the upper-bounding made in the proof. The first two follow from \cite[p. 39 (or p. 64)]{mumf1}
whereas the third one is borrowed from \cite[Proposition 3.4]{cacrxi2}.

\begin{propo} \label{prop:majoccx} Let $\F_q$ be a finite field and $X$ a curve over $\F_q$ of genus $g\geq 1$. Let $J$ be the Jacobian of $X$ and $J(\F_q)$ the rational class group.
\begin{enumerate}
\item If $q$ is odd, then $J(\F_q)[2]\leq 2^{2g}$ 
\item If $q$ is even, then $J(\F_q)[2]\leq 2^g$ 
\item Let $h$ be the class number of $X$ and, for any integer $i$ with $0\leq i \leq g-1$, $A_i$ the number of $\F_q$-rational effective divisors of degree $r$. Then 
\begin{equation*} \frac{A_i}{h}\leq\frac{g}{q^{g-i-1}(\sqrt{q}-1)^2} \end{equation*}
\end{enumerate}
\end{propo}


Let us now follow the original proof of the theorem of Cascudo \& al [only in the case $q$ even, the odd case being identic modulo using the corresponding upper-bound in Proposition \ref{prop:majoccx}].
Adding the terms $-\log_q\left(\frac{3qg}{(\sqrt{q}-1)^2} \right)$ and $2g(1-\log_q(2))$ to both sides of the inequality \eqref{eq:critqeven} reads :
$$2g+2m+2\log_q\left( \frac{3qg}{(\sqrt{q}-1)^2} \right) \leq g(1-\log_q(2)) + R -\log_q\left(\frac{3qg}{\sqrt{q}-1)^2}\right)-2 $$
Thus there exists an even integer $2d$ between the two sides of the previous inequality.
Raising $q$ to the inequalities $LHS \leq 2d$ and $2d \leq RHS$ respectively gives:
\begin{align} \frac{g}{q^{g-(2g-d+m)-1 }(\sqrt{q}-1)^2 } \leq \frac{1}{3} \label{eq:A2g-d+k}\\
						  \frac{g2^g}{q^{g-(2d-R)-1 }(\sqrt{q}-1)^2 } \leq \frac{1}{3} \label{eq:A2d-R} \end{align}   
Using the upper-bound (3) of Proposition \ref{prop:majoccx}, and combining the two inequalities \eqref{eq:A2g-d+k} and \eqref{eq:A2d-R} above with the upper-bound \ref{prop:majoccx}, yields
\begin{equation*} \label{eq:ccxmajo} h>\frac{2}{3}h\geq A_{2g-d+m} + J(\F_q)[2] A_{2d-R} \end{equation*}
Now let us choose a collection of pairwise distinct thickened points $\{P\}$ on the curve $X$ such that, for each $(d,u)$, there are exactly $n_{d,u}$ points among them of degree $d$ and multiplicity $u$ (this is possible by assumption). Let $G$ be their divisorial sum and $Q$ a closed point of degree $m$ as in the assumption. $G$ being of degree greater than $R$ by assumption \eqref{eq:thccx}, the general criterion of \cite[\S 4 Theorem 6]{cacrxi} along with the inequality \eqref{eq:ccxmajo} imply the existence of a divisor $D=X$ of degree $d$ that satisfies the following system of Riemann-Roch spaces vanishing conditions (with $K$ being the canonical divisor of $X$):
\begin{align} l(K-X+Q)=0\\ l(2X-G)=0 \end{align}  
Thus criterions (i') and (ii') of Theorem \cite[Theorem 3.5]{randJComp} are satisfied with the divisors $G$ and $D$.

		\subsection{Deriving the bounds from the previous theorem and other criterions from the litterature}
		
Let $(X_s)_s$ be a dense sequence of curves over $\F_q$ with genera $g_s$ growing to infinity, and a ratio of points of degree $r$ matching $\betad_r(q)$. Noting $\betad_r=\betad_r(q)$, this reads :
\begin{align} 
\tag{d1} & g_s  \xrightarrow[s\rightarrow \infty]{} \infty \label{eq:eqXinf}\\
\tag{d2} & B_r(X_s)=  \betad_r g_s+o(g_s) \label{eq:Xopt} \\
\tag{d3} &g_{s}=  g_{s-1}+o(g_s) \label{eq:Xdense}
\end{align}

\bigskip
Let us prove first the bound (b) in \ref{TheoA-Sym}, which generalizes \cite[Proposition 3]{bachpi}, but whose arguments were already introduced in \cite[Theorem 3.2]{bapi}. Given an integer $n$, let $s(n)$ be the smallest integer such that
\begin{equation*} \label{eq:critbp} rl B_r(X_{s(n)})-2g_{s(n)}\geq 2n+3. \end{equation*}
\eqref{eq:Xopt} makes clear (or anyway it will be in the following equivalences), that such an integer $s(n)$ exists as soon as the denominator in the criterion (b) of Theorem \ref{TheoA-Sym} is strictly positive.

Moreover $g$ being large enough, \cite[Proposition 4.3 and Remark 4.4]{bariro} state in general the existence of a zero-dimensional divisor of degree $g-5$ on $X_{s(n)}$. Thus the existence of a non-special divisor $R$ of degree (lower than) $g+3$.

Therefore, Corollary \cite[Proposition 5.1]{randJComp} applies to \eqref{eq:critbp}. Taking all $n_{d,u}$ null except $n_{r,l}$ equal to $B_r(X_{s(n)})$, this reads :
\begin{equation*} \label{eq:critbpg} \mus_q(n) \leq \mus_q(r,l) B_r(X_{s(n)}). \end{equation*}

Let us now tie the asymptotics behaviors of $g_{s(n)}$ and $B_r(X_{s(n)})$. The minimality of $s(n)$ satisfying \eqref{eq:critbp} implies :
$$rl B_r(X_{s(n)})-2g_{s(n)}\geq 2n+3>rlB_r(X_{s(n)-1})-2g_{s(n)-1}$$
Dividing the two inequalities by $g_{s(n)-1}$, and applying the asymptotic equivalences \eqref{eq:Xopt} and \eqref{eq:Xdense} (and \eqref{eq:eqXinf}) yields :
$$rl \betad_r-2+o(n)\geq \frac{2n}{g_{s(n)}}+o(1)>rl\betad_r-2 +o(n)$$
hence the asymptotic equivalence :
\begin{equation} \label{eq:ass} 2n+o(n)=(rl\betad_r-2)g_{s(n)}+o(g_{s(n)}) \end{equation}
(which implies in particular that $o(n)=o(g_{s(n)}) $). One can now divide both sides of the upper-bound \eqref{eq:critbpg} by the previous equality :
$$\frac{\mus_q(n)}{n} \leq \mus_q(r,l).2\left( \frac{\betad_r g_{s(n)} + o(n) }{(rl\betad_r-2)g_{s(n)}+o(n)} \right) $$
Multiplying and dividing the RHS parenthesis by $rl$, then subtracting and adding $2g_{s(n)}$ to the numerator of the RHS, gives the result by letting $n$ tend to infinity.  

\bigskip

The other bounds are derived similarly. Namely, given an integer $n$, consider $s(n)$ be the smallest integer such that the following inequalities hold, then apply the respective criterions with all the $n_{d,u}$ null excepted $n_{r,l}=B_r(X_{s(n)})$:
\begin{align} \label{eq:critasym} rl B_r(X_{s(n)})-g_{s(n)} & \geq 2n+5 \text{ then apply {\cite[Proposition 5.7]{randJComp}} for Theorem \ref{TheoA-Asym}} \\
\label{eq:critsym} rl B_r(X_{s(n)})-g_{s(n)} & \geq 2n+1 \text{ then apply \cite[Proposition 5.2 c)]{randJComp} for Theorem \ref{TheoA-Sym} (a)} \\
rl B_r(X_{s(n)})-g_{s(n)} & \geq 2n+1 \text{ (same $s(n)$) this time for Proposition \ref{prop:thAconjZ}}
\end{align}
[Justification for the latter: simply set $\mathrm{Cl}_0(X)(\F_q)[2]=0$ in the proof of Theorem \ref{th:ccx}, thanks to Proposition \ref{prop:notwotors}]

\begin{align}
\label{eq:critccx} rl B_r(X_{s(n)})-(1+\log_q2)g_{s(n)} & \geq 2n +3log_q\left(\frac{3qg_{s(n)}}{(\sqrt{q}-1)^2} \right) +3 \text{  if $2|q$ for Theorem \ref{TheoA-Sym} (c)} \\
\label{eq:critccx2} rl B_r(X_{s(n)})-(1+2\log_q2)g_{s(n)} & \geq 2n +3log_q\left(\frac{3qg_{s(n)}}{(\sqrt{q}-1)^2} \right) +3 \text{  otherwise for Theorem \ref{TheoA-Sym} (d)}.
\end{align}


\end{document}